%% file: conv_gdm_rev.tex
\documentclass[a4paper,final]{article}

\usepackage{epic,eepic}
\usepackage{graphicx,xcolor}
\usepackage[color,notref,notcite]{showkeys}
\usepackage{graphicx}
\usepackage{url}
\usepackage{amssymb,amsmath}
\usepackage{colortbl}
\usepackage{multicol}
\usepackage{multirow}
\usepackage{bm,subfigure}
\usepackage{paralist}

\usepackage{savesym}
\savesymbol{hyperpage}
\usepackage{hyperref}
\restoresymbol{HR}{hyperpage}
\hypersetup{
    colorlinks,
    linkcolor={red!50!black},
    citecolor={blue!50!black},
    urlcolor={blue!80!black},
  linkbordercolor={red!50!black},
}

\definecolor{labelkey}{rgb}{0.6,0,1}

\usepackage[normalem]{ulem}
\newcounter{corr}
\definecolor{violet}{rgb}{0.580,0.,0.827}
\newcommand{\corr}[3]{\typeout{Warning : a correction remains in page
\thepage}
				\stepcounter{corr}        
				{\color{blue}\ifmmode\text{\,\sout{\ensuremath{#1}}\,}\else\sout{#1}\fi}
       {\color{red}#2}
       {\color{violet} #3}}

\usepackage{constants}
\newconstantfamily{ctrcst}{symbol=C}
\def\ctel#1{\ensuremath{\Cl[ctrcst]{#1}}}
\def\cter#1{\ensuremath{\Cr{#1}}}

\newconstantfamily{ctrterm}{symbol=T}
\def\terml#1{\ensuremath{\Cl[ctrterm]{#1}}}
\def\termr#1{\ensuremath{\Cr{#1}}}

\newcommand{\mathbi}[1]{{\boldsymbol #1}}

\usepackage{amsthm}

\newtheorem{theorem}{Theorem}[section]
\newtheorem{lemma}[theorem]{Lemma}

\newtheorem{corollary}[theorem]{Corollary}
\newtheorem{remark}[theorem]{Remark}
\newtheorem{definition}[theorem]{Definition}

\newcommand{\gradD}{{\nabla_{\!\disc}}}
\newcommand{\gradDt}{{\nabla^\theta_{\!\disc}}}
\newcommand{\PiD}{{\Pi_{\!\disc}}}
\newcommand{\PiDt}{{\Pi^\theta_{\!\disc}}}

\newcommand{\ba}{\begin{array}{llll}   }
\newcommand{\bac}{\begin{array}{c}}
\newcommand{\bari}{\begin{array}{r}}
\newcommand{\ea}{\end{array}}

\newcommand{\ban}{\begin{array}{llll}}
\newcommand{\ean}{\end{array}}

\newcommand{\be}{\begin{equation}}
\newcommand{\ee}{\end{equation}}

\newcommand{\beqsys }{\beqtab \left \{ \begin{array}{l}}
\newcommand{\eeqsys }{\end{array} \right . \eeqtab }

\newcommand{\benum}{\begin{enumerate}}
\newcommand{\eenum}{\end{enumerate}}

\newcommand{\beqtab}{\begin{eqnarray}} 
\newcommand{\eeqtab}{\end{eqnarray}}

\newcommand{\dsp}{\displaystyle}

\newcommand{\bfn}{\mathbi{n}}

\newcommand{\bvarphi}{\mathbi{\varphi}}

\newcommand{\centers}{{\cal P}}
\newcommand{\cv}{K}

\renewcommand{\d}{{\rm d}}

\newcommand{\disc}{{\cal D}}
\newcommand{\dr}{\partial}
\newcommand{\dt}{{\delta\!t}}
\renewcommand{\div}{{\rm div}}

\newcommand{\edge}{\sigma}
\newcommand{\edges}{{\mathcal F}}              
\newcommand{\edgescv}{{{\edges}_\cv}}  

\newcommand{\grad}{\nabla}

\newcommand{\half}{{\frac 1 2}}

\newcommand{\interp}{{\mathcal I}}

\newcommand{\mesh}{{\mathcal M}}
\newcommand{\vertex}{\mathsf{v}}

\newcommand{\N}{\mathbb N}
\newcommand{\norm}[2]{\| #1 \|_{#2}}

\renewcommand{\O}{\Omega}
\renewcommand{\phi}{\varphi}

\newcommand{\R}{\mathbb R}

\newcommand{\velocity}{\vec{\mathbi{v}}}

\newcommand{\vertices}{\mathcal{V}}

\newcommand{\x}{\mathbi{x}}
\newcommand{\X}{\mathbi{X}}
\newcommand{\y}{\mathbi{y}}
\newcommand{\z}{\mathbi{z}}

\def\Cweak{C_{\rm w}}

\title{The gradient discretisation method for linear advection problems}
\author{J. Droniou, R. Eymard, T. Gallou\"et and R. Herbin} 

\date{\today}

\begin{document}
\maketitle

\begin{abstract} 
	We adapt the Gradient Discretisation Method (GDM), originally designed for elliptic and parabolic partial differential equations, to the case of a linear scalar hyperbolic equations. 
	This enables the simultaneous design and convergence analysis of various numerical schemes, corresponding to the methods known to be GDMs, such as finite elements (conforming or non-conforming, standard or mass-lumped), finite volumes on rectangular or simplicial grids, and other recent methods developed for general polytopal meshes. 
	The scheme is of centred type, with added linear or non-linear numerical diffusion. 
	We complement the convergence analysis with numerical tests based on the mass-lumped $\mathbb{P}_1$ conforming and non conforming finite element and on the hybrid finite volume method.
\end{abstract}
\begin{flushleft}
\end{flushleft}
\section{Introduction}
We are interested here in designing and analysing an approximation of $\bar u$, solution to the linear advection problem stated in its strong form as
\begin{subequations}
	\begin{align}
		&\partial_t \bar u+ \div (\bar u\velocity)  + \bar u q^P =  f q^I, \hbox{ in }\Omega\times (0,T), \label{eq:pbcont}
		\\
		&\bar u(\x,0) = u_{\rm ini}(\x), \hbox{ for a.e.\ }x\in\Omega, \label{stefini}
	\end{align}
	\label{pbfort}
\end{subequations}
with the following assumptions on the data:
\begin{subequations}
	\begin{align}
		& \O \mbox{ is an open bounded connected polyhedral subset of }\R^d,\ d\in\N^\star\mbox{ and }T>0, \label{hypomega} \\
		& u_{\rm ini}\in L^2(\O)\mbox{ and }f \in L^2(\O\times(0,T)), \label{hypuinif} \\
		& q^I, q^P\in L^\infty(\O\times(0,T))\mbox{ with }q^I \ge 0\mbox{  and }q^P \ge 0\mbox{ a.e.\ in }\O\times (0,T),\label{hypsource}\\
		& \velocity\in W^{1,\infty}(\O\times (0,T))^d \mbox{ satisfies }  \velocity\cdot\bm{n}=0 \hbox{ on } \dr \Omega\times(0,T)\mbox{ and }\div \velocity = q^I - q^P\mbox{ a.e.\ in }\O\times (0,T),\label{hypvelo}
\end{align}
\label{hypgen}
\end{subequations}
where $\bm{n}$ is the outer normal to $\partial\Omega$. 
Since the normal boundary value of $\velocity$ vanishes, there is no need for a boundary condition on \eqref{eq:pbcont}.

The model \eqref{pbfort} typically arises in oil recovery from underground reservoirs \cite{aziz-settari,ewing-reservoir} or in underground water resources management \cite{hydro}, in which case $q^I$ and $q^P$ may represent the injection and production wells and $\bar u$ is the concentration of injected solvent or pollutant.
The problem \eqref{pbfort} is often discretised by the upstream weighting finite volume scheme (see, for example, \cite[Chapters 5 and 6]{book} and references therein), which is easy to implement even on unstructured meshes since the problem is first order. 
There are also numerous papers studying Galerkin methods for this type of problems, which are based on the following weak formulation:
a function $\bar u$ is said to be a weak solution of Problem \eqref{pbfort} if:
	\be
		\begin{aligned}
			&\bar u\in L^2(\O\times(0,T))\mbox{ and, for all $\varphi\in C^\infty_c(\R^d\times[0,T))$,}\\
			&-\int_0^T\int_\O  \bar u\ \partial_t  \varphi \  \d \x \d t - \dsp \int_\O u_{\rm ini}(\x) \ \varphi(\x,0) \d \x +\int_0^T\int_\O(-\bar u \ \velocity \cdot\grad \varphi +  \bar u \ q^P\varphi)\d \x \d t   = \int_0^T\int_\O f  \ q^I\varphi\ \d \x\d t,
		\end{aligned}
		\label{convweakeq}
		\ee
where $C^\infty_c(\R^d\times[0,T))$ is the set of the restrictions of functions of $C^\infty_c(\R^d\times(-\infty,T))$ to  $\R^d\times[0,T)$. 

Let $0=t^{(0)}<t^{(1)}<\cdots<t^{(N)}=T$ be a discretisation of the time interval, and let $\dt^{(n+\half)}=t^{(n+1)}-t^{(n)}$.
We recall that, for $V\subset H^1(\O)$ a finite dimensional space and $\theta\in[0,1]$, the $\theta$-scheme takes the following form: $u^{(0)}\in V$ being a chosen interpolate of $u_{\rm ini}$, the scheme consists in finding, for all $n=0,\ldots,N-1$,

\be
\begin{aligned}
& u^{(n+1)}\in V,~ u^{(n+\theta)}=\theta u^{(n+1)}+(1-\theta) u^{(n)}\mbox{ and, for all $ v\in V$,}\\
&\int_\O  \dfrac {u^{(n+1)} - u^{(n)}} {\dt^{(n+\half)}}   v \  \d \x  +\int_\O(- u^{(n+\theta)} \ \velocity^{(n+\half)} \cdot\grad  v +  u^{(n+\theta)} \ (q^P)^{(n+\half)} v)\d \x    = \int_\O f^{(n+\half)}  \ (q^I)^{(n+\half)} v\ \d \x,
\end{aligned}
\label{convweakeqgalerkin}
\ee
with suitable time approximations of the data indexed by $(n+\half)$. This scheme is $L^2$ stable provided that $\theta\ge \half$, which is proved letting $ v = u^{(n+\theta)}$ and following the calculus formula 
\be
\int_\O u^{(n+\theta)} \velocity \cdot\nabla u^{(n+\theta)} \d \x=\int_\O \velocity \cdot\nabla \frac{(u^{(n+\theta)})^2}{2} \d \x = -\int_\O\frac{(u^{(n+\theta)})^2}{2} \div  \velocity\d \x.
\label{eq:calculchain}\ee
Weak convergence properties are then obtained for the approximate solution, which generally displays oscillations.
See  \cite{ern-guermond} for a complete study of the particular case of Finite element methods, and \cite{codina1998comparison} for a comparison of different Galerkin schemes.
A convergence result is proved in \cite{dunca2017optim} under strong regularity hypotheses on the solution and with a constant velocity field.

\medskip

This paper is focused on the case where the approximation of $u$ is no longer done in a subspace of $H^1(\O)$. In a number of situations, coupled problems including terms of different nature (e.g. diffusive, advective\ldots) must be solved in an industrial context where the discretisation method, imposed by the use of an existing code, is based on non conforming finite element, discontinuous Galerkin or hybrid methods (with face and cell unknowns), for example.

In order to handle such a situation, we use the Gradient Discretisation Method (GDM) framework, which gives a unified formulation of a large class of conforming and nonconforming methods; we refer the reader to the monograph \cite{gdmbook} for details. 
The idea of the GDM is to replace, in a weak formulation of the continuous problem, the continuous space by the vector space of the degrees of freedom of the method $X_\disc$, the functions $u$ and $v$ by their reconstruction  $\Pi_\disc u$ and $\Pi_\disc v$, and the gradient $\nabla v$ by the reconstruction of a discrete gradient $\nabla_\disc v$. {For conforming methods, $\Pi_\disc(X_\disc)$ is a subspace of $H^1(\Omega)$ and, for $v\in X_\disc$, $\nabla_\disc v=\nabla (\Pi_\disc v)$; for non-conforming finite element methods, $\Pi_\disc(X_\disc)$ is a space of piecewise polynomial functions and, for all $v\in X_\disc$, $\nabla_\disc v$ is the broken gradient of $\Pi_\disc v$. Discontinuous Galerkin methods, which are popular in the framework of hyperbolic problems, can also be embedded in the GDM; for these methods, $\Pi_\disc(X_\disc)$ is again a space of piecewise polynomial functions, the expression of $\nabla_\disc v$ takes into account both the broken gradient of $\Pi_\disc v$ and the jump terms, and no additional stabilisation term has to be introduced in the formulation of the scheme (see \cite[Chapter 11]{gdmbook}). Note that for fully discrete methods or mass-lumped versions of the previous schemes, $\Pi_\disc$ is a genuine function reconstruction (see the schemes used in Section \ref{sec:num}).}

A natural scheme would then be: given an interpolate $u^{(0)}\in X_\disc$ of $u_{\rm ini}$, solve for $n=0,\ldots,N-1$,
\be
\begin{aligned}
& u^{(n+1)}\in X_{\disc},~ u^{(n+\theta)}=\theta u^{(n+1)}+(1-\theta) u^{(n)}\mbox{ and, for all $ v\in X_{\disc}$,}\\
&\int_\O  \Pi_\disc \dfrac {u^{(n+1)} - u^{(n)}} {\dt^{(n+\half)}}  \Pi_\disc  v \  \d \x   \\ 
&\qquad+\int_\O(- \Pi_\disc u^{(n+\theta)} \ \velocity^{(n+\half)} \cdot\grad_\disc  v +  \Pi_\disc u^{(n+\theta)} \ (q^P)^{(n+\half)}\Pi_\disc  v)\d \x    = \int_\O f^{(n+\half)}  \ (q^I)^{(n+\half)}\Pi_\disc  v\ \d \x.
\end{aligned}
\label{convweakeqgdm}
\ee
Unfortunately, it does not seem possible to establish the stability (and thus the convergence) of \eqref{convweakeqgdm} due to the absence of the equivalent of the calculus chain \eqref{eq:calculchain}  in this fully discrete setting involving function and gradient reconstructions $\Pi_\disc$ and $\nabla_\disc$ instead of the classical differential operators.
To obtain a scheme amenable to a convergence analysis, we thus consider an alternative formulation, using a skew-symmetric reformulation of the advective term.

\medskip

If $\grad \bar u \in L^2(\O\times(0,T))$, owing to the relation
\[
\frac12\div\velocity + q^P=\frac12\left(q^I-q^P\right)+ q^P=\frac 1 2 (q^I+q^P),
\]
a function $\bar u \in L^2(\O\times(0,T))$ is a solution to \eqref{convweakeq} if and only if it satisfies
\be\begin{aligned}
&\forall v\in C^\infty_c(\R^d\times[0,T)),\\
& -\int_0^T\int_\O  \bar u \ \partial_t   v \  \d \x \d t - \dsp \int_\O u_{\rm ini}(\x)\  v(\x,0) \ \d \x \\ 
&\qquad +\int_0^T\int_\O\left(\half\grad\bar u\cdot \velocity v-\half\bar u\velocity\cdot\grad  v + \half\bar u(q^I + q^P)  v\right)\d \x \d t  = \int_0^T\int_\O f  q^I v\d \x\d t.
\end{aligned}
\label{convweakeqreg}\ee

The idea to discretise \eqref{eq:pbcont} is then to mimick the formulation \eqref{convweakeqreg} instead of \eqref{convweakeq} in the discrete setting  (this idea is in the same line as the weak formulation chosen in \cite[Hypothesis (A1)]{bur2010exp}). 
Indeed, similarly to the standard skew-symmetric formulation of the convective term in the Navier-Stokes equations, the advection component in \eqref{convweakeqreg} vanishes when the solution is taken as a test function. 
The GDM scheme based on \eqref{convweakeqreg} is thus: take $u^{(0)}\in X_\disc$ and interpolant of $u_{\rm ini}$ and, for all $n=0,\ldots,N-1$,
\be
\begin{aligned}
& u^{(n+1)}\in X_{\disc},~ u^{(n+\theta)}=\theta u^{(n+1)}+(1-\theta) u^{(n)}\mbox{ and, for all $ v\in X_{\disc}$,}\\
&\int_\O  \Pi_\disc \dfrac {u^{(n+1)} - u^{(n)}} {\dt^{(n+\half)}}  \Pi_\disc  v \  \d \x  +\int_\O\Big( \half \grad_\disc u^{(n+\theta)} \cdot \velocity^{(n+\half)} \Pi_\disc  v -\half \Pi_\disc u^{(n+\theta)} \velocity^{(n+\half)}\cdot\grad_\disc  v 
\\ &\qquad\qquad\qquad+\half \Pi_\disc u^{(n+\theta)} \left[(q^I)^{(n+\half)}+ (q^P)^{(n+\half)}\right] \Pi_\disc  v\Big)\d \x    = \int_\O f^{(n+\half)}  \ (q^I)^{(n+\half)}\Pi_\disc  v\ \d \x.
\end{aligned}
\label{convweakeqgdmskew}
\ee
Letting $ v = u^{(n+\theta)}$ in \eqref{convweakeqgdmskew} leads to an estimate on $ \Pi_\disc u^{(n+\theta)}$, which entails a weak convergence property for the reconstruction of the function. 
However a new difficulty arises: the scheme \eqref{convweakeqgdmskew} does not yield any estimate on $\grad_\disc u^{(n+\theta)}$; this prevents us from obtaining any limit (even weak) for this term, and thus from passing to the limit to recover the continuous problem.

This issue is solved by introducing a stabilisation term that yields a weak bound on $\grad_\disc u^{(n+\theta)}$. 
Several versions of such a stabilisation term can be found \cite{john2018fin,franca1992stab}, such as the symmetric linear stabilisation of \cite{bur2010exp}, or the Streamline-Upwind/Petrov-Galerkin (SUPG) stabilisation \cite{bochev2004stab,knobloch2008defin,hugues1987fin,carmo2003stab}. 
The latter is equivalent to replacing, in the term $\bar u\velocity$ of \eqref{eq:pbcont}, $\bar u$ by $\bar u - h \frac {\velocity}{|\velocity|}\cdot\nabla\bar u$ (this is a kind of continuous upstream weighting for a mesh with size $ h$). This leads to the term
\[
 \div \left(\left[\bar u -  h \frac {\velocity}{|\velocity|}\cdot\nabla\bar u\right]\velocity\right) = \div (\bar u \velocity-  h\Lambda \nabla\bar u),\hbox{ with }\Lambda(\x,t)  =  \frac {\velocity(\x,t)}{|\velocity(\x,t)|}\otimes \velocity(\x,t).
\]
It is then numerically more stable to complete the SUPG scheme by modifying $\Lambda$ into
\[
 \Lambda(\x,t)  =  \frac {\velocity(\x,t)}{|\velocity(\x,t)|}\otimes \velocity(\x,t) + \mu \mathrm{Id},
\]
for a small value $\mu>0$. 
This choice of stabilisation term $\div (-h\Lambda \nabla\bar u)$ can be generalised into
\begin{equation}\label{term:plap}
- h^\alpha\div(\Lambda |\nabla \bar u|_\Lambda^{p-2}\nabla \bar u)\mbox{ where }|\nabla \bar u|_\Lambda = \sqrt{\Lambda \nabla \bar u\cdot \nabla \bar u},
\end{equation}
for some $p\in (1,+\infty)$ and $\alpha>0$, and $\Lambda(\x,t)$ symmetric positive definite with uniformly  bounded eigenvalues. 
 An obvious and easy choice is $p=2$ and $\Lambda = \rm{Id}$, which leads to the classical Laplace operator. 
However, using $p \ne 2$ may lead to a smaller numerical diffusion, see Section \ref{sec:num};  let us note that in this case, the linear model \eqref{pbfort} is approximated by a non-linear problem, which is not in general much of a problem, since the complete coupled physical model usually involves other non-linear terms.
In this paper, we stabilise the scheme \eqref{convweakeqgdmskew} by introducing the discrete version of the stabilisation term \eqref{term:plap}, which leads to Scheme \eqref{eq:gradscheme}.
Since the GDM method also includes meshless schemes, the stabilisation term depends on a parameter $h_\disc$ which is an adaptation to the hyperbolic setting of the space size of gradient discretisation for elliptic problems, see Definition \ref{def:sizegd} below. 

\medskip

{In addition to providing a generic formulation that applies to a large variety of schemes, this paper presents the following original features:
\begin{enumerate}
 \item The analysis applies to mesh-based as well as meshless schemes, owing to Definition \ref{def:sizegd} of the size of gradient discretisation which gives us a way to introduce an intrinsic vanishing viscosity without referring to any mesh size. 
 \item We study and compare, for different values of $p$, the effect of the stabilisation of a hyperbolic scheme by $p$-Laplace vanishing diffusion. Numerical examples show that in some cases, values of $p$ different from 2 lead to more accurate solutions.
 \item The strong convergence of the stabilised scheme is obtained through an energy estimate, proved in the $L^2$ framework by regularisation as in \cite{MR1022305};
this energy estimate is also used for the proof of uniqueness of the solution.  
\item Convergence is established without assuming additional regularity on the solution or the velocity field, and a uniform-in-time weak convergence is proved.
\end{enumerate}
}

This paper is organised as follows. 
The continuous problem and the energy estimate are studied in Section \ref{sec:pbcont}.  
We then apply in Section \ref{sec:gdm} the gradient discretisation tools to Problem \eqref{convweakeq}, and derive some estimates which are used in Section \ref{sec:cv} to establish the convergence of the scheme; as a by-product of this convergence, we also obtain an existence result for the solution to \eqref{convweakeq}.
In Section \ref{sec:num} some numerical results are provided, using three different schemes that fit into the GDM framework.
 
\section{The continuous problem}\label{sec:pbcont}
Since the flux is null on the boundary $\partial \Omega$, the problem \eqref{convweakeq} may be reformulated on the whole space $\R^d$ by extending $\velocity$, $q^I$ and $q^P$ to $\R^d\times\R$:  we first choose an extension $\velocity\in  W^{1,\infty}(\R^d\times \R)^d$, and then set $q^I=\max(\div\velocity, 0)$ and $q^P=\max(-\div\velocity, 0)$ outside $\Omega\times (0,T)$.
We also extend $\bar u$, $f$ and $u_{\rm ini}$ by the value $0$ outside $\O\times(0,T)$ and $\O$ respectively. 
With these extensions and {assuming} \eqref{hypgen}, the problem \eqref{convweakeq} is equivalent to the following problem, posed on the whole space:
\be
	\begin{aligned}
		&\bar u\in L^2(\R^d\times(0,T))\mbox{ and, for all $\varphi\in C^\infty_c(\R^d\times[0,T))$,}\\
		&-\int_0^T\int_{\R^d}  \bar u\ \partial_t  \varphi \  \d \x \d t - \dsp \int_{\R^d} u_{\rm ini}(\x) \ \varphi(\x,0) \d \x \\ 
		&\qquad+\int_0^T\int_{\R^d}(-\bar u \ \velocity \cdot\grad \varphi +  \bar u \ q^P\varphi)\d \x \d t   = \int_0^T\int_{\R^d} f  \ q^I\varphi\ \d \x\d t.
	\end{aligned}
	\label{eq:pbcontext}
\ee

\begin{lemma}[Weak continuity with respect to time]\label{lem:wcontime}
{Assuming} \eqref{hypgen}, let $\bar u$ be a solution of \eqref{convweakeq}, or to \eqref{eq:pbcontext} after extending $\bar u$ by 0 outside of $\Omega$.
Let $\psi\in C^\infty_c(\R^d)$. Then the function $\bar U_\psi~:~ t\mapsto \int_{\R^d} \bar u(\x,t) \psi(\x)\ \d \x$ satisfies $\bar U_\psi\in H^{1}(0,T)\subset C^0([0,T])$ and $\bar U_\psi(0) = \int_{\R^d}  u_{\rm ini}(\x)\psi(\x)\ \d \x$. 
Hence, $\bar u \in \Cweak([0,T],L^2(\O))$, where $\Cweak([a,b],L^2(\O))$ stands for the space of functions $[a,b]\to L^2(\O)$ that are continuous weakly in  $L^2(\O)$. 
\end{lemma}
\begin{proof}
	Let $\Theta\in C^\infty_c([0,T))$. 
	Taking $\varphi(\x,t)=\Theta(t)\psi(\x)$ in \eqref{eq:pbcontext} yields 
	\be
		\begin{aligned}
			&-\int_0^T\Theta'(t) \bar U_\psi(t)  \d t - \dsp \Theta(0) \int_{\R^d} u_{\rm ini}(\x) \ \psi(\x) \d \x \\ 
			&+\int_0^T\Theta(t)\int_{\R^d}(-\bar u(\x,t) \ \velocity(\x,t) \cdot\grad \psi(\x) +  \bar u(\x,t) \ q^P(\x,t)\psi(\x))\d \x \d t \\
			&\qquad= \int_0^T\Theta(t)\int_{\R^d} f(\x,t)  \ q^I(\x,t)\psi(\x)\ \d \x\d t.
		\end{aligned}
		\label{eq:wcontime}
	\ee
	Restricting to $\Theta\in C^\infty_c(0,T)$ this shows that, in the weak derivative sense,
	\begin{equation}\label{Upsi:der}
 		\bar U_\psi'(t) = \int_{\R^d} \Big((f(\x,t)  \ q^I(\x,t) -  \bar u(\x,t) \ q^P(\x,t))\psi(\x)\ + \bar u(\x,t) \ \velocity(\x,t) \cdot\grad \psi(\x) \Big)\d \x.
	\end{equation}
	Since the right hand side of the above equation belongs to $L^2(0,T)$, this concludes the proof that $\bar U_\psi\in H^{1}(0,T)\subset C^0([0,T])$.
	The relation $\bar U_\psi(0) = \int_{\R^d}  u_{\rm ini}(\x)\psi(\x)\ \d \x$ is proved taking $\Theta\in C^\infty_c([0,T))$ such that $\Theta(0)=1$ in \eqref{eq:wcontime}, integrating-by-parts in time and using \eqref{Upsi:der}. 
\end{proof}

\begin{lemma}[Energy estimate]\label{lem:enerest}
	{Assuming} \eqref{hypgen}, any solution $\bar u$ of \eqref{convweakeq} satisfies:
	\be
		\begin{aligned}
			& \half\int_0^T\int_\O  (\bar u(\x,t))^2 \d \x \d t  + \half \int_0^T(T-t)\int_\O (q^I(\x,t) + q^P(\x,t)) \ \bar u(\x,t)^2\d \x \d t  \\ \dsp
			&\qquad=  \frac T 2\int_\O u_{\rm ini}(\x)^2 \d \x + \int_0^T(T-t)\int_\O f(\x,t)   q^I(\x,t)\bar u(\x,t)\d \x\d t.
		\end{aligned}
		\label{convweakeqrenorm}
	\ee
\end{lemma}
\begin{proof}
	By density of $C^\infty_c(\R^d\times[0,T))$ in $H^{1}(\R^d\times(0,T))$, we can consider functions $\varphi\in C^0_c(\R^d\times[0,T))\cap H^{1}(\R^d\times(0,T))$ in \eqref{eq:pbcontext}. 
Letting $\rho$ be a mollifier on $\R^d$, and $\rho_n(\x) = n^d\rho(n\x)$ for all $\x\in\R^d$ and $n\in\N^\star$, we choose the function $\varphi$ defined by
\[
\forall \x,t\in \R^d\times[0,T],\ \varphi(\x,t) = (T-t)\int_{\R^d} \int_{\R^d}\bar u(\z,t)\rho_n(\y - \z)\rho_n(\y - \x)\d\z\d\y,
\]
which satisfies $\varphi\in C^0_c(\R^d\times[0,T))\cap H^{1}(\R^d\times[0,T))$ owing to Lemma \ref{lem:wcontime}. 
Using an integration by parts with respect to $\y$, we notice that
\[
 \forall \x,t\in \R^d\times[0,T],\ \grad \varphi(\x,t) = (T-t)\int_{\R^d} \int_{\R^d}\bar u(\z,t)\grad\rho_n(\y - \z)\rho_n(\y - \x)\d\z\d\y,
\]
With this choice of $\varphi$ in \eqref{eq:pbcontext} leads to $\terml{t1}^{(n)} + \terml{t2}^{(n)} +\terml{t3}^{(n)} +\terml{t4}^{(n)} +\terml{t4a}^{(n)}+\terml{t5}^{(n)}  = \terml{t6}^{(n)}$, with
\be\begin{aligned}
&\termr{t1}^{(n)} = \int_0^T\int_{\R^d}  \bar u(\x,t) \int_{\R^d} \int_{\R^d}\bar u(\z,t)\rho_n(\y - \z)\rho_n(\y - \x)\d\z\d\y \d \x \d t, \\
&\termr{t2}^{(n)} = -\int_0^T(T-t)\int_{\R^d}  \bar u(\x,t) \int_{\R^d} \partial_t\Big(\int_{\R^d}\bar u(\z,t)\rho_n(\y - \z)\rho_n(\y - \x)\d\z\Big)\d\y \d \x \d t, \\
&\termr{t3}^{(n)} = -T \int_{\R^d} u_{\rm ini}(\x)\int_{\R^d} \int_{\R^d}u_{\rm ini}(\z) \rho_n(\y - \z)\rho_n(\y - \x)\d\z\d\y \d \x,  \\
&\termr{t4}^{(n)} =  -\int_0^T(T-t)\int_{\R^d}\bar u(\x,t) \ \velocity(\y,t)\cdot\int_{\R^d} \int_{\R^d}\bar u(\z,t)\grad\rho_n(\y - \z)\rho_n(\y - \x)\d\z\d\y \d \x \d t,  \\ 
&\termr{t4a}^{(n)} =  \int_0^T(T-t)\int_{\R^d}(-\bar u(\x,t) \ (\velocity(\x,t)- \velocity(\y,t))\cdot \int_{\R^d} \int_{\R^d}\bar u(\z,t)\grad\rho_n(\y - \z)\rho_n(\y - \x)\d\z\d\y \d \x \d t,  \\ 
&\termr{t5}^{(n)} =  \int_0^T(T-t)\int_{\R^d}\bar u(\x,t) \ q^P(\x,t)\int_{\R^d} \int_{\R^d}\bar u(\z,t)\rho_n(\y - \z)\rho_n(\y - \x)\d\z\d\y\d \x \d t,  \\
&\termr{t6}^{(n)} =  \int_0^T(T-t)\int_{\R^d} f(\x,t)  \ q^I(\x,t)\int_{\R^d} \int_{\R^d}\bar u(\z,t)\rho_n(\y - \z)\rho_n(\y - \x)\d\z\d\y\ \d \x\d t.
\end{aligned}
\label{eq:termn}\ee
Introducing the function $\bar u_n(\y,t) = \int_{\R^d}\bar u(\z,t)\rho_n(\y - \z)\d\z$, which converges to $\bar u$ in $L^2(\R^d\times(0,T))$ as $n\to\infty$ and satisfies $\bar u_n\in H^1(\R\times(0,T))$ and $\bar u_n(\y,0) = \int_{\R^d} u_{\rm ini}(\z)\rho_n(\y - \z)\d\z$ (see Lemma \ref{lem:wcontime}), we have
\[
 \termr{t1}^{(n)} = \int_0^T\int_{\R^d}  \bar u_n(\y,t)^2\d\y \d t
\]
and, using an integration-by-parts,
\begin{align*}
 \termr{t2}^{(n)} = -\int_0^T(T-t)\int_{\R^d}  \bar u_n(\y,t) \partial_t \bar u_n(\y,t)\d\y \d t 
={}& -\int_0^T(T-t)\int_{\R^d}   \partial_t \left(\frac 1 2 \bar u_n(\y,t)^2\right)\d\y \d t\\
={}& \frac T 2  \int_{\R^d} \bar u_n(\y,0)^2\d\y  -\int_0^T\int_{\R^d}    \frac 1 2 \bar u_n(\y,t)^2\d\y \d t.
\end{align*}
Gathering these results leads to
\[
 \termr{t1}^{(n)} + \termr{t2}^{(n)} +\termr{t3}^{(n)} =  \frac 1 2\int_0^T\int_{\R^d}    \bar u_n(\y,t)^2\d\y \d t - \frac T 2  \int_{\R^d} \bar u_n(\y,0)^2\d\y 
\]
and therefore
 \[
  \lim_{n\to\infty} (\termr{t1}^{(n)} + \termr{t2}^{(n)} +\termr{t3}^{(n)}) =\frac 1 2 \int_0^T\int_{\R^d}  \bar u(\x,t)^2 \d \x \d t - \frac T 2\int_{\R^d}  u_{\rm ini}(\x)^2 \d \x.
 \]
Turning to $\termr{t4}^{(n)}$ we write, using the divergence formula and $\div\velocity=q^I-q^P$,
\begin{align*}
 \termr{t4}^{(n)} ={}&  -\int_0^T(T-t)\int_{\R^d}\bar u_n(\y,t) \ \velocity(\y,t)\cdot\grad\bar u_n(\y,t)\d\y \d t \\
={}& -\int_0^T(T-t)\int_{\R^d} \velocity(\y,t)\cdot\grad\left(\frac 1 2\bar u_n(\y,t)^2\right)\d\y \d t\\
 ={}& \frac 1 2\int_0^T(T-t)\int_{\R^d}\bar u_n(\y,t)^2(q^I(\y,t)-q^P(\y,t))\d\y \d t.
\end{align*}
Hence,
 \[
  \lim_{n\to\infty} \termr{t4}^{(n)} = \frac 1 2\int_0^T(T-t)\int_{\R^d}\bar u(\y,t)^2(q^I(\y,t)-q^P(\y,t))\d\y \d t.
 \]
We then easily see that, as $n\to\infty$,
\[
\termr{t5}^{(n)} = \int_0^T(T-t)\int_{\R^d}\bar u_n(\y,t)^2 \ q^P(\y,t)\d \y \d t
\to \int_0^T(T-t)\int_{\R^d}\bar u(\y,t)^2 \ q^P(\y,t)\d \y \d t
\] 
and
\[
  \termr{t6}^{(n)} \to \int_0^T(T-t)\int_{\R^d}\bar u(\y,t) f(\y,t) \ q^I(\y,t)\d \y \d t.
\]
The proof is completed by gathering all the above convergence results and by proving that
 \be
  \lim_{n\to\infty} \termr{t4a}^{(n)}= 0.
 \label{eq:limtermt4a}\ee
In order to do so, we follow the technique of \cite[Lemma II.1]{MR1022305} and \cite[Lemma B.4]{eym2010conv}. 
An integration-by-parts gives
\[
 \termr{t4a}^{(n)} =  \int_0^T(T-t) \int_{\R^d} a_n(\y,t)\bar u_n(\y,t) \d\y \d t,  
\]
with
\[
 a_n(\y,t) = \div\Big(\int_{\R^d}(\bar u(\x,t) \rho_n(\y - \x)\ (\velocity(\x,t)- \velocity(\y,t))\d \x\Big). 
\]
Since the function $(T-t)\bar u_n$ converges to $(T-t)\bar u$ in $L^2(\R^d\times(0,T))$ as $n\to\infty$,
the proof of \eqref{eq:limtermt4a} is complete if we can show that $a_n\to 0$ weakly in $L^2(\R^d\times(0,T))$. We have
\begin{equation}\label{an.cut}
 a_n(\y,t) = \int_{\R^d}\bar u(\x,t) \grad\rho_n(\y - \x)\cdot(\velocity(\x,t)- \velocity(\y,t))\d \x  - \int_{\R^d}\bar u(\x,t) \rho_n(\y - \x)\div \velocity(\y,t))\d \x.
\end{equation}
By Lipschitz continuity of $\velocity$, there exists $C_{\velocity}>0$ depending only on $\velocity$ such that $|\nabla\rho_n(\y-\x)\cdot(\velocity(\x,t)- \velocity(\y,t))|\le C_{\velocity} |\y - \x|\,|\nabla \rho_n(\y-\x)|$. 
Noting that the sequence of functions $\z\mapsto |\z|\,|\nabla\rho_n(\z)|$ is bounded in $L^1(\R^d)$, Young's inequality for convolution shows that the first term in the right-hand side of \eqref{an.cut} is bounded in $L^2(\R^d\times (0,T))$. 
The same Young inequality also easily shows that the second term in this right-hand side is also bounded in the same space, which proves that $a_n$ itself remains bounded in  $L^2(\R^d\times(0,T))$. 
The weak convergence of $a_n$ therefore only needs to be assessed for smooth functions.
Taking $\psi\in C^\infty_c(\R^d\times(0,T))$, we have
\[
 \int_0^T\int_{\R^d} a_n(\y,t)\psi(\y,t) \d\y \d t = - \int_0^T\int_{\R^d}\int_{\R^d}\bar u(\y+\z,t) \rho_n(\z)\ (\velocity(\y+\z,t)- \velocity(\y,t))\cdot\grad\psi(\y,t)\d \z\d\y \d t.
\]
Hence, using the Lipschitz continuity of $\velocity$ and the fact that $\rho_n$ is supported in the ball centred at $0$ and of radius $1/n$, there exists $C>0$ depending only on $\bar u, \velocity$ and $\psi$ such that
\[
 \left|\int_0^T\int_{\R^d} a_n(\y,t)\psi(\y,t) \d\y \d t\right|\le \frac C n.
\]
Hence $a_n$ converges to $0$ weakly in $L^2(\R^d\times(0,T))$, which concludes  the proof of \eqref{eq:limtermt4a} and of the lemma.
\end{proof}

\begin{corollary}[Uniqueness]\label{lem:uniq}
{Assuming} \eqref{hypgen}, there exists at most one solution $\bar u$ to \eqref{convweakeq}.
\end{corollary}
\begin{proof}
 The difference of two solutions to \eqref{convweakeq} is a solution for the same problem with right-hand-side $f=0$ and initial condition $u_{\rm ini} = 0$. 
 The energy estimate \eqref{convweakeqrenorm} shows that this difference is a.e.\ equal to 0.
\end{proof}

\section{The gradient discretisation method for the linear advection equation}\label{sec:gdm}

The gradient discretisation method (GDM) is a general framework for nonconforming approximations of elliptic or parabolic problems, see \cite{gdmbook} for a general presentation of the method and of some models and schemes it applies to.

The principle of the GDM is to design a set of discrete elements (space, operators) called a gradient discretisation (GD), which is substituted in the weak formulation of the PDE \textit{in lieu} of the related continuous elements leading to a discretisation scheme. 

\begin{definition}[Gradient discretisation]\label{defgraddisc}
Let $p\in (1,+\infty)$ be given and let $p'\in (1,+\infty)$ with $\frac 1 p + \frac 1 {p'} = 1$. A gradient discretisation $\disc$ is defined by $\disc =(X_{\disc},\Pi_\disc,\nabla_\disc)$ where:
\begin{enumerate}
\item the set of discrete unknowns $X_{\disc}$ is a finite dimensional vector space on $\R$,
\item the linear mapping $\Pi_\disc~:~X_{\disc}\to L^{\max(2,p')}(\O)$ reconstructs functions,
\item the linear mapping $\nabla_\disc~:~X_{\disc}\to L^{\max(2,p)}(\O)^d$ reconstructs approximations of their gradients, 
\item the quantity $\Vert \cdot \Vert_{\disc} := \Vert \Pi_\disc \cdot \Vert_{L^2(\O)}+\Vert \nabla_\disc \cdot \Vert_{L^p(\O)^d}$ defines a norm on $X_{\disc}$.
\end{enumerate}
\end{definition}
\begin{remark} 
In the above definition, the definition of the norm is not standard in the GDM setting (in the sense of \cite[Definition 2.1]{gdmbook}), because of the simultaneous use of the $L^p$, $L^{p'}$ and $L^2$ norms. 
\end{remark}

This notion is extended to evolution problems in the following definition.

\begin{definition}[Space-time gradient discretisation]\label{adisct}
A family $\disc^T = (X_{\disc}, \Pi_\disc,\nabla_\disc, \interp_\disc,(t^{(n)})_{n=0,\ldots,N})$ is a space-time gradient discretisation if
\begin{itemize}
\item $\disc=(X_{\disc}, \Pi_\disc,\nabla_\disc)$ is a gradient discretisation of $\O$, in the sense of Definition \ref{defgraddisc},
\item $\interp_\disc~:~L^2(\O)\to X_{\disc}$ is an interpolation operator,
\item $t^{(0)}=0<t^{(1)}\ldots<t^{(N)}=T$.
\end{itemize}
We then set $\dt^{(n+\half)} = t^{(n+1)} -t^{(n)}$, for $n=0,\ldots,N-1$, and 
$\dt_\disc = \max_{n=0,\ldots,N-1} \dt^{(n+\half)}$.
\end{definition}

The properties of GDs are assessed through the two following functions $S_{\disc}$ and $W_{\disc}$. The first one measures an interpolation error:
\be\begin{aligned}
&S_{\disc}:W^{2,\infty}(\O)\to [0,+\infty)\mbox{ such that, for $\varphi\in W^{2,\infty}(\O)$},\\
&S_{\disc}(\varphi) = \min_{v\in X_{\disc}}\left(\Vert \Pi_\disc v - \varphi\Vert_{L^{\max(2,p')}(\O)} + \Vert \nabla_\disc v -
\nabla\varphi\Vert_{L^{\max(2,p)}(\O)^d}\right),
\end{aligned}
\label{defsdisc}\ee
whilst the second one is a measure of a conformity defect ({\it i.e.} the defect in a discrete integration-by-parts formula): letting $W^{1,\infty}_{\bm{n},0}(\O)^d$ be the set of elements of $W^{1,\infty}(\O)^d$ with zero normal trace on $\partial\O$,
\be
\begin{aligned}
&W_{\disc}:W^{1,\infty}_{\bm{n},0}(\O)^d\to [0,+\infty)\mbox{ such that, for $\bvarphi\in W^{1,\infty}_{\bm{n},0}(\O)^d$},\\
& W_{\disc}(\bm{\varphi}) = \max_{u\in X_{\disc}\setminus\{0\}}\frac 1 {\Vert  u \Vert_{\disc}}\left\vert
\int_\O \left(\gradD u(\x)\cdot\bvarphi(\x) + \Pi_\disc u(\x) \div\bvarphi(\x)\right)  \d\x \right\vert.
\end{aligned}
\label{defwdisc}\ee

Let us now define the space size of a GD relative to some regularity spaces. {This definition, which holds for both mesh-based and meshless methods, is a measure of the approximation properties of a given GD (this notion is defined in the framework of elliptic problems with homogeneous boundary conditions in \cite[Definition 2.22]{gdmbook})}. 

\begin{definition}[Space size of a GD]\label{def:sizegd}
Let $\disc$ be a gradient discretisation. The space-size of $\disc$ is $h_\disc$ defined by
\begin{align}
&h_\disc = \max\left(\sup\bigg\{  \dsp \frac {S_\disc(\varphi)} {\Vert \varphi\Vert_{W^{2,\infty}(\O)} }\,:\,\varphi\in W^{2,\infty}(\O)\setminus\{0\}\bigg\};
\sup\bigg\{ \dsp \frac {W_\disc(\bm{\varphi})} {\Vert \bm{\varphi}\Vert_{W^{1,\infty}_{\bm{n},0}(\O)^d} }\,:\,\bm{\varphi}\in W^{1,\infty}_{\bm{n},0}(\O)^d \setminus\{0\} \bigg\}\right).\label{eq:defhdisc}
\end{align}
\end{definition}

\begin{remark}[Link between $h_\disc$ and the size of the mesh for mesh-based GDs]
 {
In the case of the mesh-based GDs detailed in \cite[Chapters 8-14]{gdmbook}, $h_{\disc}$ is related to the size of the mesh $h_\mesh$ by $h_\disc\le C h_\mesh$ (see, e.g., \cite[Remark 2.24]{gdmbook}).}

\end{remark}

\begin{definition}[Consistent and limit-conforming sequence of space-time gradient discretisation]\label{def:conslimconf}~\\
A sequence $(\disc^T_m)_{m\in\N}$ of space-time gradient discretisations is said to be consistent and limit-conforming if $h_{\disc_m}$, $\dt_{\disc_m}$ and, for all $w\in L^2(\O)$, $\Vert w -\Pi_{\disc_m} \interp_{\disc_m} w\Vert_{L^2(\O)}$ tend to 0 as $m\to\infty$.
\end{definition}

\begin{remark}[Link with the core properties of a GD in the framework of elliptic or parabolic problems]

 {An adaptation of \cite[Lemma 2.25]{gdmbook} to elliptic problems with homogeneous Neumann boundary conditions yields an equivalence between Definition \ref{def:conslimconf} and \cite[Definitions 3.4 and 3.5]{gdmbook} of consistent and limit-conforming sequences of  gradient discretisations, assuming that the sequence  $(\disc^T_m)_{m\in\N}$ is compact (this holds true for the GDs detailed in \cite[Chapters 8-14]{gdmbook}).
}
 
\end{remark}

Given a space--time gradient discretisation $\disc^T = (X_{\disc}, \Pi_\disc,\nabla_\disc,\interp_\disc,(t^{(n)})_{n=0,\ldots,N})$ (in the sense of Definition \ref{adisct}), we now describe the gradient scheme defined from this GD. 
For $n=0,\ldots,N-1$ and a given space-time function $g\in L^1(\O\times(0,T))^\ell$ with $\ell=1$, $\ell=d$ or $\ell=d\times d$ ($g$ could be $\Lambda$, $f$, $\velocity$, $q^I$ or $q^P$), set, for a.e.\ $\x\in\O$ and for all $n=0,\ldots,N-1$,
\begin{equation}\label{def:gdiscn}
g^{(n+\half)}(\x) = \frac 1 {\dt^{(n+\half)}} \int_{t^{(n)}}^{t^{(n+1)}} g(\x,t) \d t \hbox{ and }g_\disc(\x,t)=g^{(n+\half)}(\x)\hbox{ for a.e. }t\in (t^{(n)},t^{(n+1)}).
\end{equation}
Let $\theta\in [\frac12,1]$ and $\alpha\in (0,p)$. The ($\theta$-implicit) scheme for Problem \eqref{convweakeq} is defined by replacing the continuous space and operators in \eqref{convweakeqreg} with their discrete counterparts given by $\disc$, as follows: find $u=(u^{(n)})_{n=0,\ldots,N}$ such that
\be
\left\{
\ba
u^{(0)} = \interp_\disc u_{\rm ini}\mbox{ and, for $n = 0,\ldots,N-1$, $u^{(n+1)} \in X_{\disc}$ is such that,}\\
\mbox{setting }\delta_\disc^{(n+\half)} u 
   = \Pi_\disc \dfrac {u^{(n+1)} - u^{(n)}} {\dt^{(n+\half)}} \mbox{ and } u^{(n+\theta)}=\theta u^{(n+1)}+(1-\theta) u^{(n)},
\\
\dsp \int_\O 
\left( \delta_\disc^{(n+\half)} u  \Pi_\disc v  +
  \half \grad_\disc u^{(n+\theta)} \cdot \velocity^{(n+\half)} \Pi_\disc v -\half \Pi_\disc u^{(n+\theta)} \velocity^{(n+\half)}\cdot\grad_\disc v  \right.
\\
\dsp 
\qquad\left. 
+\half \Pi_\disc u^{(n+\theta)} \left[ (q^I)^{(n+\half)}+  (q^P)^{(n+\half)}\right] \Pi_\disc v+ h_\disc^\alpha|\gradD u^{(n+\theta)}|_\Lambda^{p-2} \Lambda^{(n+\half)}\gradD u^{(n+\theta)} \cdot\gradD v
\right)
\d\x 
\\ 
\dsp \qquad\qquad= 
\int_\O f^{(n+\half)}     (q^I)^{(n+\half)}  \Pi_\disc v  \d\x,
 \qquad \forall v\in X_{\disc},
\ea
\right.
\label{eq:gradscheme}
\ee
denoting for short
\be
|\gradD u^{(n+\theta)}|_\Lambda = \sqrt{\Lambda^{(n+\half)}\gradD u^{(n+\theta)}\cdot\gradD u^{(n+\theta)}}.
\label{eq:norlambda}\ee
We introduce the following notations $\Pi^\theta_\disc$ and $\gradDt$ for reconstructed space-time functions: given $v=(v_n)_{n=0,\ldots,N}$ in $X_\disc^{N+1}$, we set 
\be\begin{aligned}
&\Pi_\disc^\theta v(\x,t) = \Pi_\disc v^{(n+\theta)}(\x)\mbox{ and }  \gradDt v(\x,t) = \gradD v^{(n+\theta)}(\x)\,,\\
&\hbox{for a.e.\ }(\x,t)\in \O\times(t^{(n)},t^{(n+1)}],\ \forall n=0,\ldots,N-1.
\end{aligned}
\label{defuappt}\ee
We extend these definitions to $t=0$ by setting $\Pi_\disc v(\x,0)=\Pi_\disc v_0$ and $\gradD v(\x,0)=\gradD v_0$.

\section{Convergence analysis}\label{sec:cv}

Our main convergence result is stated in the following theorem. 
We recall (see \cite[Definition 2.11]{DE15}) that a sequence $(v_n)_{n\in\N}$ of bounded functions $[0,T]\to L^2(\O)$ is said to converge uniformly on $[0,T]$ weakly in $L^2(\Omega)$ towards a function $v$ if for all $\phi\in L^2(\O)$, the sequence of functions $t\in [0,T]\mapsto \langle v_n(t),\phi\rangle_{L^2(\O)}\in\R$ converges uniformly on $[0,T]$ towards the function $t\mapsto \langle v(t),\phi\rangle_{L^2(\O)}$.

\begin{theorem}[Convergence of the GDM]\label{cvgce}
{Assuming} \eqref{hypgen}, let  $(\disc^T_m)_{m\in\N}$ be a consistent and limit-conforming sequence of space-time gradient discretisations in the sense of Definition \ref{def:conslimconf}.  Let $\theta\in [\half,1]$, $p\in(1,+\infty)$  and $\alpha\in (0,p)$ be given. Then, for any  $m\in\N$, there exists a unique $u_{m}$ solution to Scheme \eqref{eq:gradscheme} with $\disc^T=\disc_m^T$.

Moreover, as $m\to\infty$, $\Pi_{\disc_m}^\theta u_m$ converges in $L^2(\O\times(0,T))$, and uniformly on $[0,T]$ weakly in $L^2(\O)$, to the unique solution $\bar u$ of Problem \eqref{convweakeq}.
\end{theorem}
\begin{remark}[Theoretical order of convergence]
{Assuming sufficient smoothness of the continuous solution and the velocity field, and letting $p=2$,
it seems possible to derive a theoretical error estimate in $L^\infty(0,T;L^2(\Omega))$ norm with order $h_\disc^{\min(\alpha,2-\alpha)}$. 
This provides a maximal order 1 if $\alpha=1$.
However, in the numerical tests with a regular solution (see Section \ref{sec:resnumcase2}), much better numerical
orders of convergence are obtained, even letting $\alpha = 2$. The question of the theoretical derivation
of these better rates remains open.}
\end{remark}
The uniqueness component of this theorem is the most straightforward part, and the purpose of the following lemma.

 \begin{lemma}[Uniqueness of a discrete solution]  \label{lem:uniquediscrete}
  {Assuming} \eqref{hypgen}, let $\disc^T = (X_{\disc}, \Pi_\disc,\nabla_\disc,\interp_\disc,(t^{(n)})_{n=0,\ldots,N})$ be a space-time gradient discretisation in the sense of Definition \ref{adisct}.  Let $\theta\in [\half,1]$, $p\in(1,+\infty)$  and $\alpha\in (0,p)$ be given. Then there exists at most one solution to Scheme \eqref{eq:gradscheme}. 
\end{lemma}

\begin{proof}
The scheme defines exactly one approximation $u^{(0)}$. Let us assume that, for a given $n\in\N$ and for a given $u^{(n)}$, there exist two solutions $u^{(n+1)}$ and $\widehat u^{(n+1)}$ to  Scheme \eqref{eq:gradscheme}. Let us create the difference of the two equations \eqref{eq:gradscheme}, and let us choose $v = u^{(n+\theta)} - \widehat u^{(n+\theta)} = \theta (u^{(n+1)}-\widehat u^{(n+1)})$ in the resulting equation. We obtain
\be
\ba
\dsp \int_\O 
\left( \left[\frac {\theta}  {\dt^{(n+\half)}}    + \frac {\theta^2} {2} \ ( (q^I)^{(n+\half)}+  (q^I)^{(n+\half)})\right]  (\Pi_\disc u^{(n+1)} - \Pi_\disc \widehat u^{(n+1)} )^2\right.
\\
\dsp 
\qquad\left.  + h_\disc^\alpha\Lambda^{(n+\half)}(|\gradD u^{(n+\theta)}|_\Lambda^{p-2} \gradD u^{(n+\theta)} - |\gradD \widehat u^{(n+\theta)}|_\Lambda^{p-2} \gradD \widehat u^{(n+\theta)} )\cdot(\gradD u^{(n+\theta)} - \gradD \widehat u^{(n+\theta)} )
\right)
\d\x 
= 0.
\ea
\label{eq:gradschemedif}
\ee
It is classical (see for instance \cite[Lemma 2.40]{gdmbook} or \cite[Lemma 2.1]{bar1993fin}), that  
\[
 \forall \xi, \chi \in\R^d,\ (|\xi|^{p-2} \xi - |\chi|^{p-2} \chi )\cdot(\xi - \chi ) \ge \underline{\lambda}\min(\frac{p-1} 2,2^{1-p})  |\xi - \chi|^2\ (|\xi|+|\chi|)^{p-2}.
 \]
Applying this inequality in \eqref{eq:gradschemedif} with $\chi = (\Lambda^{(n+\half)})^{1/2} \gradD u^{(n+\theta)}$ and  $\xi = (\Lambda^{(n+\half)})^{1/2} \gradD \widehat u^{(n+\theta)}$ (in which the left hand side is therefore the sum of non-negative terms), we get that $\gradD u^{(n+\theta)} = \gradD \widehat u^{(n+\theta)}$ a.e., and therefore $\gradD u^{(n+1)} = \gradD \widehat u^{(n+1)}$ as well as $\Pi_\disc u^{(n+1)} = \Pi_\disc \widehat u^{(n+1)}$. Hence, thanks to the property of the norm assumed in Definition \ref{defgraddisc}, $u^{(n+1)}=\widehat u^{(n+1)}$, which concludes the proof of uniqueness by induction.

\end{proof}

The proof of Theorem \ref{cvgce} hinges on \emph{a priori} estimates stated in the following lemma.

 \begin{lemma}[$L^\infty(0,T;L^2(\O))$ and discrete $L^2(0,T;H^1_0(\O))$ estimates, existence of a discrete solution]  \label{estimldlp}
  {Assuming} \eqref{hypgen}, let $\disc^T = (X_{\disc}, \Pi_\disc,\nabla_\disc,\interp_\disc,(t^{(n)})_{n=0,\ldots,N})$ be a space-time gradient discretisation in the sense of Definition \ref{adisct}.  Let $\theta\in [\half,1]$, $p\in(1,+\infty)$  and $\alpha\in (0,p)$ be given. Then there exists one and only one solution to Scheme \eqref{eq:gradscheme}. Moreover, this solution satisfies, for all $k=1,\ldots,N$,
	\be
   \begin{aligned}
     \label{convweakeqhatdisc}
     \int_\O {}&\left( \half \Pi_\disc u^{(k)}(\x)^2 -  \half \Pi_\disc u^{(0)}(\x)^2\right)\d\x
       + h_\disc^\alpha\int_0^{t^{(k)}} \int_\O |\gradDt  u (\x,t)|_\Lambda^p\d\x \d t  
       \\
				&+ \half \int_0^{t^{(k)}}\int_\O  \Pi_\disc^\theta u(\x,t)^2 \left[q^I_\disc(\x,t)+q^P_\disc(\x,t)\right]\d\x \d t
       \le  \int_0^{t^{(k)}}\int_\O  f_\disc(\x,t)   q^I_\disc(\x,t) \Pi_\disc^\theta  u (\x,t) \d\x\d t,
   \end{aligned}
	\ee
and there exists $\ctel{estimld}>0$, depending only on $C_{\rm ini} \ge \Vert u_{\rm ini} -\Pi_\disc \interp_\disc u_{\rm ini}\Vert_{L^2(\O)}$, $\underline{\lambda}$, $f$ and $q^I$ such that
  \begin{equation} 
   \Vert\PiDt  u\Vert_{L^\infty(0,T;L^2(\O))} \le \cter{estimld}
   \label{estimfldeux}
  \end{equation}
and
  \begin{equation}
  h_\disc^\alpha \Vert\gradDt u \Vert_{L^p(\O\times(0,T))}^p \le  \cter{estimld}.
   \label{estimhund0}
  \end{equation}
\end{lemma}

\begin{remark}[Weak $BV$ estimate] 
The estimate \eqref{estimhund0} is the adaptation in the GDM framework of the classical weak $BV$ estimate used for finite volumes see \cite{cha-92-wea} for the seminal paper and \cite[chapters 5 \& 6]{book} for more general results. 
This estimate is used in two occasions: first to pass to the limit in the skew-symmetric term, and second to show that the stabilisation term vanishes at the limit. 	
\end{remark}

\begin{proof}
Before establishing the existence of at least one discrete solution to Scheme \eqref{eq:gradscheme}, let us first prove that any solution to this scheme satisfies \eqref{convweakeqhatdisc}--\eqref{estimhund0}. 
We first notice that for all $a,b\in\R$,  
\begin{align*}
(a-b)(\theta a + (1-\theta)b) ={}& 
(a-b)\left[\left(\theta-\half\right) a + \left(\half-\theta\right)b\right]
+\half (a-b)(a+b)\\
={}& \left(\theta-\half\right)(a-b)^2+ \half (a^2 - b^2)\ge
\half (a^2 - b^2).
\end{align*}
Hence, letting $v = \dt^{(n+\half)} u^{(n+\theta)}$ in \eqref{eq:gradscheme} and applying the
estimate above with $a=\Pi_\disc u^{(n+1)}$ and $b=\Pi_\disc u^{(n)}$, we obtain
\begin{multline*}
\dsp \int_\O 
\Bigg( 
   \half (\Pi_\disc u^{(n+1)})^2 - \half (\Pi_\disc u^{(n)})^2
+\half \dt^{(n+\half)}(\Pi_\disc u^{(n+\theta)})^2 \left[ (q^I)^{(n+\half)}+  (q^I)^{(n+\half)}\right]\\
 + \dt^{(n+\half)}h_\disc^\alpha\ |\gradD u^{(n+\theta)}|_\Lambda^p \Bigg)
\d\x 
\le \dt^{(n+\half)}
\int_\O f^{(n+\half)}     (q^I)^{(n+\half)}  \Pi_\disc u^{(n+\theta)}  \d\x.
\end{multline*}
Taking $k=1,\ldots,N$ and summing this inequality over $n=0,\ldots,k-1$ proves \eqref{convweakeqhatdisc}. 

The Young inequality and the property $0\le  q^I_\disc \le  q^I_\disc +  q^P_\disc$ yield
\begin{align*}
f_\disc  q^I_\disc \Pi_\disc^\theta u 
\le 
  \dfrac{1}{2} (\Pi_\disc^\theta u)^2  \left[ q^I_\disc+  q^P_\disc\right]+ \dfrac{1}2 (f_\disc)^2   q^I_\disc.
\end{align*}
Plugging this into \eqref{convweakeqhatdisc} leads to
\begin{align}
\frac12\norm{\Pi_\disc u^{(k)}}{L^2(\O)}^2 + h_\disc^\alpha \underline{\lambda}^{p/2} \norm{\gradDt u}{L^p(\O\times (0,t^{(k)}))}^p
\le{}& \frac12\norm{\Pi_\disc \interp_\disc u_{\rm ini}}{L^2(\O)}^2 + \frac12 \int_0^{t^{(k)}}\int_\O (f_\disc)^2   q^I_\disc\d\x \d t\nonumber\\
\le{}& \frac12\norm{\Pi_\disc \interp_\disc u_{\rm ini}}{L^2(\O)}^2 + \norm{f}{L^2(\O\times(0,T))}
\norm{q^I}{L^\infty(\O\times(0,T))},
\label{interm}
\end{align}
where we have used the Jensen inequality to bound the $L^2$-norm of $f_\disc$ by the $L^2$-norm of $f$.
Estimate \eqref{estimhund0} directly follows from \eqref{interm} with $k=N$. 
Estimate \eqref{estimfldeux} is also a consequence of \eqref{interm}, once we notice that $\Pi_\disc^\theta u(\x,t)=\theta \Pi_\disc u^{(n+1)}(\x)+(1-\theta)\Pi_\disc u^{(n)}(\x)$ for a.e.\ $\x\in\O$, all $t\in (t^{(n)},t^{(n+1)})$ and all $n=0,\ldots,N-1$.

\medskip

We can now prove the existence of a solution to Scheme \eqref{eq:gradscheme} (the uniqueness is proved in Lemma \ref{lem:uniquediscrete}). If $p=2$ then, at each time step, \eqref{eq:gradscheme} describes a linear square system on $u^{(n+\theta)}$ (after substituting $u^{(n+1)}=u^{(n)}+\theta^{-1}(u^{(n+\theta)}-u^{(n)})$). The estimates \eqref{estimfldeux} and \eqref{estimhund0} show that any solution $u^{(n+\theta)}$ to this system satisfies \emph{a priori} bounds. The kernel of the matrix of this linear system is therefore reduced to $\{0\}$, and the matrix is invertible, which establishes the existence of a unique solution $u^{(n+\theta)}$ (and thus of $u^{(n+1)}$) to the system at time step $n+1$.

If $p\neq 2$ we use the topological degree \cite{deimling}. Let us assume the existence of $u^{(n)}$. 
Let us substitute the term $|\gradD u^{(n+\theta)}|_\Lambda^{p-2}\gradD u^{(n+\theta)}$ of the scheme by $\nu|\gradD u^{(n+\theta)}|_\Lambda^{p-2}\gradD u^{(n+\theta)}+(1-\nu)\gradD u^{(n+\theta)}$ for $\nu\in [0,1]$. 
It is clear that the above estimates still hold (again after substituting $u^{(n+1)}=u^{(n)}+\theta^{-1}(u^{(n+\theta)}-u^{(n)})$) so that $\norm{\Pi_\disc u^{(n+1)}}{L^2(\O)}$ and $\nu\Vert\gradD u^{(n+1)} \Vert_{L^p(\O)}^p +
  (1-\nu) h_\disc^\alpha \Vert\gradD u^{(n+1)} \Vert_{L^2(\O)}^2 $ remain bounded independently of $\nu$. 
  We infer from this latter estimate a bound on $\Vert u^{(n+1)} \Vert_{\disc}$ that is uniform with respect to $\nu$. 
  Hence, all solutions to the scheme with the above substitution remain bounded independently of $\nu$. 
  This shows that, on a large enough ball, the topological degree of the non-linear mapping defining the scheme is independent of $\nu$. 
  For $\nu=0$ this mapping is linear and the arguments developed in the case $p=2$  show that its topological degree is non-zero. 
  The degree for the original scheme (corresponding to $\nu=1$) is therefore also non-zero, proving that this scheme has at least one solution. \end{proof}

We can now prove our convergence results, starting with the uniform-in-time weak-in-space convergence.

\begin{proof}[Proof of Theorem \ref{cvgce}: uniform-in-time weak-in-space convergence]
Owing to \eqref{estimfldeux} there is $\bar u\in L^2(\O\times(0,T))$ and a subsequence of  $(\disc_m)_{m\in\N}$ such that $\Pi^\theta_{\disc_m} u_m$ converges to $\bar u$ in $L^\infty(0,T;L^2(\O))$ weak-$\star$ as $m\to\infty$.
Let $m\in \N$, and let us denote $\disc = \disc_m$ (belonging to the above subsequence); we drop some indices $m$ to simplify the notations.

Let $\varphi\in C^\infty_c([0,T))$ and $w\in C^\infty_c(\R^d)$, and let $P_m w\in X_\disc$ that realises the minimum in $ S_{\disc}(w)$. 
We denote by $P_m \varphi:(0,T)\to\R$ the function equal to $\varphi^{(n+1-\theta)}:=\theta \varphi(t^{(n)})+(1-\theta) \varphi(t^{(n+1)})$, on $(t^{(n)},t^{(n+1)})$, for all $n=0,\ldots,N-1$.

For $n=0,\ldots,N-1$ and $t\in (t^{(n)},t^{(n+1)})$, let $u^{(\theta)}_m(t)=\theta u_m^{(n+1)}+(1-\theta)u_m^{(n)}\in X_\disc$, and notice that $\PiDt u_m(t)=\Pi_\disc u_m^{(\theta)}(t)$ and $\gradDt u_m(t)=\gradD u_m^{(\theta)}(t)$.
By definition \eqref{defwdisc} of $W_\disc$ and \eqref{eq:defhdisc} of $h_\disc$, since
$w\velocity(t) \in W^{1,\infty}_{\bm{n},0}(\O)^d$ we have, for a.e.\ $t\in(0,T)$, recalling the definition \eqref{eq:defhdisc} of $h_\disc$,
\[
 \left|\int_\O ( \gradD u^{(\theta)}_m(t)\cdot \velocity(t) w + \PiD u^{(\theta)}_m(t)\div(w\velocity(t)))\d\x\right| \le W_\disc(w\velocity(t)) \Vert u^{(\theta)}_m(t)\Vert_{\disc}\le h_\disc\Vert w\velocity\Vert_{W^{1,\infty}(\O\times(0,T))^d}\Vert u^{(\theta)}_m(t)\Vert_{\disc}.
\]
Thanks to \eqref{estimfldeux}--\eqref{estimhund0}, there is $\cter{estimldd}$ depending only on $\cter{estimld}$ and $T$ such that
\[
h_{\disc}\left(\int_0^T\Vert u^{(\theta)}_m(t)\Vert_{\disc}^p\right)^{1/p}\le  \ctel{estimldd}(h_{\disc}+h_{\disc}^{1-\frac{\alpha}{p}}).
\]
This right-hand side tends to $0$ as $m\to\infty$ (remember that $\alpha<p$) and thus, since $P_m\varphi$ is bounded in $L^\infty(0,T)$,
\[
 \lim_{m\to\infty} \int_0^T P_m \varphi(t)\int_\O  (\gradDt u_m(t)\cdot \velocity(t) w + \PiDt u_m(t)\div(w\velocity(t)))\d\x\d t =0.
\]
By strong convergence of $P_m \varphi$ to $\varphi$ in $L^2(0,T)$
and weak convergence of $\PiDt u_m$ to $\bar u$ in $L^2(\O\times(0,T))$ we infer
\[
 \lim_{m\to\infty} \int_0^T\int_\O P_m \varphi(t) \gradDt u_m(t)\cdot \velocity(t) w \d\x\d t= -\int_0^T\int_\O \varphi(t)\bar u\div(w\velocity(t))\d\x\d t.
\]
A Cauchy--Schwarz inequality yields
\begin{multline*}
  \left|\int_0^T P_m\varphi(t)\int_\O  (\gradDt u_m(t)\cdot \velocity(t) w -  \gradDt u_m(t)\cdot \velocity(t)\Pi_{\disc} P_m  w)\d\x\d t\right|\\
 \le \Vert \gradDt u_m\Vert_{L^p(\O\times(0,T))^d} 
  \Vert P_m\varphi(\velocity w - \velocity\Pi_{\disc} P_m  w)\Vert_{L^{p'}(\O\times(0,T))^d}
\end{multline*}
and, by definition of $P_m\varphi$, $h_\disc$ and $P_m w$,
\[
 \Vert P_m\varphi(t)(\velocity w - \velocity\Pi_{\disc} P_m  w)\Vert_{L^{p'}(\O\times(0,T))^d} \le  
T^{1/p}\norm{\varphi}{L^\infty(0,T)}\norm{\velocity}{L^\infty(\O\times(0,T))^d}h_\disc \norm{w}{W^{2,\infty}(\O)}.
\]
Therefore, using \eqref{estimhund0} again,
\begin{align}
\lim_{m\to\infty} \int_0^T P_m\varphi(t)\int_\O  \gradDt u_m(t)\cdot \velocity(t)\Pi_{\disc} P_mw   \d\x\d t={}& \lim_{m\to\infty} \int_0^T P_m\varphi(t)\int_\O  \gradDt u_m(t)\cdot \velocity(t) w \d\x\d t \nonumber\\
={}& -\int_0^T\varphi(t)\int_\O \bar u(t)\div(w\velocity(t))\d\x\d t.
\label{eq:weakbvlim}\end{align}

We take $\dt^{(n+\half)} \varphi(t^{(n)}) P_m w$ as test function in \eqref{eq:gradscheme} and sum the resulting equation over $n=0,\ldots,N-1$. This gives
\be\label{sum.terms}
\termr{1m}^{(m)} +\half \termr{2m}^{(m)}+\half\termr{div}^{(m)}-\half\termr{vm}^{(m)} + \termr{grad}^{(m)} = \termr{3m}^{(m)} 
\ee
with
\begin{align*}
\terml{1m}^{(m)} ={}&  \sum_{n=0}^{N-1} \dt^{(n+\half)}\varphi^{(n+1-\theta)}\int_\O \delta_\disc^{(n+\half)}  u_m \ \Pi_\disc P_m w\d\x,\\
\terml{2m}^{(m)} ={}& \sum_{n=0}^{N-1} \dt^{(n+\half)}\varphi^{(n+1-\theta)}\int_\O \grad_\disc u_m^{(n+\theta)} \cdot \velocity^{(n+\half)} \Pi_\disc P_m w\d\x,\\
\terml{div}^{(m)} ={}& \sum_{n=0}^{N-1} \dt^{(n+\half)}\varphi^{(n+1-\theta)} \int_\O \Pi_\disc u_m^{(n+\theta)} \left[( (q^I)^{(n+\half)}+  (q^I)^{(n+\half)}\right] \Pi_\disc P_m w\d\x,\\
\terml{vm}^{(m)} ={}& \sum_{n=0}^{N-1} \dt^{(n+\half)}\varphi^{(n+1-\theta)}\int_\O \Pi_\disc u_m^{(n+\theta)}\velocity_\disc(\x,t)\cdot\grad_\disc P_m w\d\x,\\
\terml{grad}^{(m)} ={}& h_\disc^\alpha\sum_{n=0}^{N-1} \dt^{(n+\half)}\varphi^{(n+1-\theta)}\int_\O |\gradD u_m^{(n+\theta)}|_\Lambda^{p-2} \Lambda^{(n+\half)}\gradD u_m^{(n+\theta)}\cdot\grad_\disc P_m w\d\x,\\
\terml{3m}^{(m)} ={}& \sum_{n=0}^{N-1}\varphi^{(n+1-\theta)} \dt^{(n+\half)} \int_\O f^{(n+\half)}     (q^I)^{(n+\half)}  \Pi_\disc  P_m w\d\x\d t.
\end{align*}
The summation-by-parts formula \cite[Eq. (D.17)]{gdmbook} reads
\[
 \sum_{n=0}^{N-1} (b^{(n+1)} - b^{(n)})(\theta a^{(n)} + (1-\theta)a^{(n+1)}) = - b^{(0)}a^{(0)} -  \sum_{n=0}^{N-1} (\theta b^{(n+1)} + (1-\theta)b^{(n)})(a^{(n+1)} - a^{(n)}) + b^{(N)}a^{(N)}.
\]
Using this relation to transform, in the sum appearing in $\termr{1m}^{(m)}$, the term $ \dt^{(n+\half)}\varphi^{(n+1-\theta)}\delta_\disc^{(n+\half)}  u_m$ into $(\varphi(t^{(n)})-\varphi(t^{(n+1)}))\Pi_\disc u_m^{(n+\theta)}$, we see that
$$
\termr{1m}^{(m)} = -\int_0^T\varphi'(t) \int_\O\PiDt u_m\ \Pi_\disc P_m w\d\x\d t - \varphi(0) \int_\O\Pi_\disc u_m^{(0)}\ \Pi_\disc P_m w\d\x,
$$
and so, since $\PiDt u_m\to \bar u$ weakly in $L^2(\O\times(0,T))$, $\PiD P_m w\to w$ strongly in $L^2(\O)$, and $\Pi_\disc u_m^{(0)}=\Pi_{\disc_m} \interp_{\disc_m} u_{\rm ini}\to u_{\rm ini}$ in $L^2(\Omega)$,
$$
\lim_{m\to\infty}\termr{1m}^{(m)} = -\int_0^T\varphi'(t) \int_\O u(\x,t) w\d\x\d t - \varphi(0) \int_\O u_{\rm ini}w\d\x.
$$
Noticing that
\[
\termr{2m}^{(m)} = \int_0^T P_m \varphi\int_\O \grad_\disc^\theta u_m \cdot \velocity \Pi_\disc P_m w\d\x,
\]
the relation \eqref{eq:weakbvlim} yields
$$
\lim_{m\to\infty}\termr{2m}^{(m)} = -\int_0^T\varphi(t)\int_\O \bar u\div(w\velocity)\d\x\d t =  -\int_0^T\varphi(t)\int_\O \bar u w \div(\velocity)\d\x\d t -\int_0^T\varphi(t)\int_\O \bar u \velocity\cdot\grad w\d\x\d t.
$$
Moreover, since 
\be\label{conv.divvel}
q^I_{\disc_m} - q^P_{\disc_m}\to q^I - q^P\mbox{ a.e.\ in $\Omega\times(0,T)$ as $m\to\infty$ and remains bounded},
\ee
$\PiDt u$ weakly converges
to $u$ in $L^2(\O\times(0,T))$, $\PiD (P_m w)$ strongly converges in $L^2(\O)$, and $\gradD (P_m w)$ strongly converges to $\nabla w$
in $L^2(\O)^d$, we have
\begin{align*}
\lim_{m\to\infty}\termr{div}^{(m)} ={}& \int_0^T\varphi(t)\int_\O \bar u w (q^I + q^P) \d\x\d t,\\
\lim_{m\to\infty}\termr{vm}^{(m)} ={}& \int_0^T\varphi(t)\int_\O \bar u \velocity\cdot \grad w\d\x\d t.
\end{align*}
The H\"older inequality and \eqref{estimhund0} show that
\begin{align*}
|\termr{grad}^{(m)}|\le{}& \overline{\lambda}^{p/2} T^{1/p}\norm{\varphi}{L^\infty(0,T)}h_\disc^\alpha \norm{\gradDt u}{L^p(\O\times(0,T))^d}^{p-1}\norm{\gradD (P_m w )}{L^p(\O)} \\
\le{}& \overline{\lambda}^{p/2} T^{1/p}\norm{\varphi}{L^\infty(0,T)}h_\disc^{\alpha/p} \cter{estimld}^{(p-1)/p}\norm{\gradD (P_m w )}{L^p(\O)}.
\end{align*}
The boundedness of $\gradD (P_m w )$ in $L^p(\O)$ (since this sequence converges in this space) and $\alpha>0$
 then yield $\lim_{m\to\infty}\termr{grad}^{(m)} = 0$.

Finally, using \eqref{conv.divvel} again,
$$
\lim_{m\to\infty}\termr{3m}^{(m)} = \int_0^T\varphi(t) \int_\O f(\x,t)   q^I(\x,t)  w(\x)\d\x\d t.
$$

Passing to the limit $m\to\infty$ in \eqref{sum.terms} shows that $\bar u$ satisfies \eqref{convweakeq} for any test function of the form $\varphi(t)w(\x)$, and thus for sums of such test functions.
Since the set $\mathcal T=\{\sum_{i=1}^q\varphi_i(t) w_i(\x)\,:\,
q\in\N,\varphi_i\in C^\infty_c[0,T),w_i\in C^\infty_c(\R^d)\}$ is dense in the set of the restrictions to $\overline{\O}\times [0,T)$ of the elements of $C^\infty_c(\R^d\times[0,T))$, we conclude that $\bar u$ is a solution of \eqref{convweakeq}.

\medskip

It now suffices to prove the uniform-in-time weak-$L^2(\O)$ convergence of $\Pi_{\disc_m}^\theta u_m$ to $\bar u$.
Let $w\in C^\infty_c(\R^d)$ and define $P_m w$ as before. For  $0\le s\le t$, writing $\Pi_{\disc_m}^\theta u_m(\x,t)- \Pi_{\disc_m}^\theta u_m(\x,s)$ as the sum of its jumps at each $t^{(n)}\in (s,t)$ (see \cite[Proof of Theorem 4.19]{gdmbook} for details), Scheme \eqref{eq:gradscheme} and the estimates in Lemma \ref{estimldlp} give the existence of $\ctel{cquatre}$, depending only on the data introduced in \ref{hypgen}, such that
\[
\left|\int_\O(\Pi_{\disc_m}^\theta u_m(\x,t)- \Pi_{\disc_m}^\theta u_m(\x,s))\Pi_{\disc_m} (P_mw)(\x)\d\x \right| \le
\cter{cquatre} (t-s + \dt_m)^{1/2} \Vert P_mw \Vert_{\disc_m}.
\]
Hence, introducing $\pm w$, and using \eqref{estimfldeux} again, 
\[
\left|\int_\O (\Pi_{\disc_m}^\theta u_m(\x,t)- \Pi_{\disc_m}^\theta u_m(\x,s))w(\x)\d\x\right| 
 \le
(t-s + 2\dt_m)^{1/2} \cter{cquatre} \Vert P_mw \Vert_{\disc_m} + 2 C_1\Vert w - \Pi_{\disc_m} (P_mw)\Vert_{L^2(\O)}.
 \]

Using $\sqrt{t-s+2\dt_m}\le\sqrt{t-s}+\sqrt{2\dt_m}$, we get
\[
\left|\int_\O(\Pi_{\disc_m}^\theta u_m(\x,t)- \Pi_{\disc_m}^\theta u_m(\x,s))w(\x)\d\x \right\vert\le g(t-s,h^\varphi_m),
\]
with $g(a,b) = \sqrt{a}\ctel{ccinq} + b$, $\cter{ccinq}  = \cter{cquatre}\sup_m \Vert P_mw \Vert_{\disc_m}$ and $h^\varphi_m = (2\dt_m)^{1/2} \cter{ccinq} + 2C_1 \Vert w - \Pi_{\disc_m} (P_mw)\Vert_{L^2(\O)}$. We then may apply \cite[Theorem C.11]{gdmbook} or \cite[Theorem 6.2]{DE15} to deduce that  $\Pi_{\disc_m}^\theta u_m$  weakly tends to $u$  in $L^2(\Omega)$ uniformly on $[0,T]$.
\end{proof}

\begin{proof}[Proof of Theorem \ref{cvgce}: strong convergence]
The proof makes use of the continuous energy estimate \eqref{convweakeqrenorm} and a discrete version thereof, in a similar way as in the proof of \cite[Theorem 2.16]{DE15}. Let us first establish this discrete energy estimate. We remark that for all $a,b,c,d\in\R$,  
\begin{align*}
(a-b)\left(\frac {a+b} 2 + \alpha \frac {a-b} 2\right)&\left(T - \frac {c+d} 2 - \alpha \frac {c-d} 2\right) \\
={}& \half(T-d)a^2 - \half(T-c)b^2 + \half (d-c)\left(\frac {a+b} 2 + \alpha \frac {a-b} 2\right)^2 \\
&+ \frac {(a-b)^2} 2 \left[\alpha\left(T - \frac {c+d} 2 - \alpha \frac {c-d} 2\right) + \frac {d-c} 4(1-\alpha^2)\right].
\end{align*}
Setting  $t^{(n+1-\theta)} = \theta t^{(n)} + (1-\theta)t^{(n+1)}$, letting $v = \dt^{(n+\half)} u^{(n+\theta)} (T - t^{(n+1-\theta)})$ in \eqref{eq:gradscheme}, applying the above
relation with $a=\Pi_\disc u^{(n+1)}$, $b=\Pi_\disc u^{(n)}$, $\alpha = 2\theta - 1$, $c = t^{(n)}$, $d = t^{(n+1)}$, and dropping the last addend (which is positive), we obtain
\begin{multline}
\dsp \int_\O 
\Bigg( 
   \half (\Pi_\disc u^{(n+1)})^2 (T - t^{(n+1)}) - \half (\Pi_\disc u^{(n)})^2 (T - t^{(n)}) + \half \dt^{(n+\half)}(\Pi_\disc u^{(n+\theta)})^2\\
+ (T - t^{(n+1-\theta)})\half \dt^{(n+\half)}(\Pi_\disc u^{(n+\theta)})^2 \left[ (q^I)^{(n+\half)}+  (q^I)^{(n+\half)}\right]\\
+ (T - t^{(n+1-\theta)})\dt^{(n+\half)}h_\disc^\alpha\ |\gradD u^{(n+\theta)}|_\Lambda^p \Bigg)
\d\x 
\\ 
\le (T - t^{(n+1-\theta)})\dt^{(n+\half)}
\int_\O f^{(n+\half)}     (q^I)^{(n+\half)}  \Pi_\disc u^{(n+\theta)}  \d\x.
\label{eq:estimmult}
\end{multline}
Summing the obtained inequality on $n=0,\ldots,N-1$, and denoting by $t_\disc$ the function equal to $t^{(n+1-\theta)}$ for all $t\in (t^{(n)},t^{(n+1)})$, we get
\begin{multline}
\dsp 
   \half \int_0^T\int_\O (\PiDt u(\x,t))^2\d\x\d t 
+ \half \int_0^T(T - t_\disc(t)) \int_\O (\PiDt u(\x,t))^2 \left[q^I_\disc(\x,t)+ q^P_\disc(\x,t)\right] \d\x\d t
\\ 
\le \frac T 2\int_\O (\Pi_\disc\interp_\disc u_{\rm ini}(\x))^2\d\x + \int_0^T(T - t_\disc(t))\int_\O  f_\disc(\x,t)     q^I_\disc(\x,t)  \PiDt u(\x,t)  \d\x\d t.
\label{eq:estimmulttt}
\end{multline}
Taking the superior limit as $m\to\infty$ of the above inequality for $\disc = \disc_m$, we get
\begin{multline}
\dsp 
   \half \limsup_{m\to\infty}\int_0^T\int_\O (\Pi_{\disc_m}^\theta u_m(\x,t))^2\left(1+ (T - t_{\disc_m}(t))\left[q^I_{\disc_m}(\x,t)+q^P_{\disc_m}(\x,t)\right]\right)\d\x\d t 
\\ 
\le \frac T 2\int_\O u_{\rm ini}(\x)^2\d\x + \int_0^T(T - t)\int_\O  f(\x,t)     q^I(\x,t)  \bar u(\x,t)  \d\x\d t.
\label{eq:estimmulttttt}
\end{multline}
We then use \eqref{convweakeqrenorm} to substitute the right-hand side of this inequality and find
\be
   \limsup_{m\to\infty}\int_0^T\int_\O (\Pi_{\disc_m}^\theta u_m)^2\left(1+ (T - t_{\disc_m})\left[q^I_{\disc_m}+q^P_{\disc_m}\right]\right)\d\x\d t 
\le  \int_0^T\int_\O \bar u^2\left(1+(T-t)\left[q^I+q^P\right]\right) \d\x\d t.
\label{eq:limsup}\ee
Developing the square $(\Pi_{\disc_m}^\theta u_m - \bar u)^2$ we have
\begin{align}
 \int_0^T\int_\O (\Pi_{\disc_m}^\theta u_m - \bar u)^2{}&\left(1+(T - t_{\disc_m})\left[q^I_{\disc_m}+q^P_{\disc_m}\right]\right) \d\x\d t\nonumber\\
 ={}& \int_0^T\int_\O (\Pi_{\disc_m}^\theta u_m)^2\left(1+ (T - t_{\disc_m})\left[q^I_{\disc_m}+q^P_{\disc_m}\right]\right)\d\x\d t \nonumber\\
 &-2 \int_0^T \int_\O \Pi_{\disc_m}^\theta u_m \ \bar u \left(1+(T - t_{\disc_m})\left[q^I_{\disc_m}+q^P_{\disc_m}\right]\right) \d\x\d t\nonumber\\
 &+ \int_0^T \int_\O \bar u^2 \left(1+(T - t_{\disc_m})\left[q^I_{\disc_m}+q^P_{\disc_m}\right]\right) \d\x\d t.
\label{cv:L2strong.0}
\end{align}
The limit of the second (resp. third) term in the right-hand side is obtained by weak/strong (resp. strong) convergence:
\[
 \lim_{m\to\infty}\int_0^T \int_\O \Pi_{\disc_m}^\theta u_m \ \bar u \left(1+(T - t_{\disc_m})\left[q^I_{\disc_m}+q^P_{\disc_m}\right]\right) \d\x\d t =  \int_0^T\int_\O \bar u^2  \left(1+(T-t)\left[q^I+q^P\right]\right) \d\x\d t,
\]
and
\[
 \lim_{m\to\infty}\int_0^T \int_\O \bar u^2 \left(1+(T - t_{\disc_m})\left[q^I_{\disc_m}+q^P_{\disc_m}\right]\right) \d\x\d  t =  \int_0^T\int_\O \bar u^2 \left(1+(T-t)\left[q^I+q^P\right]\right) \d\x\d t.
\]
Hence, using \eqref{eq:limsup} to deal with the first term in the right-hand side of \eqref{cv:L2strong.0} we find
\[
  \limsup_{m\to\infty}\int_0^T\int_\O (\Pi_{\disc_m}^\theta u_m - \bar u)^2 \left(1+(T - t_{\disc_m})\left[q^I_{\disc_m}+q^P_{\disc_m}\right]\right) \d\x\d t \le 0,
\]
and therefore, since $1+(T - t_{\disc_m})\left[q^I_{\disc_m}+q^P_{\disc_m}\right]\ge 1$,
\[
  \lim_{m\to\infty}\int_0^T\int_\O (\Pi_{\disc_m}^\theta u_m - \bar u)^2 \d\x\d t = 0,
\]
which concludes the proof of the convergence of $ \Pi_{\disc_m}^\theta u_m$ to $\bar u$ in $L^2(\O\times(0,T))$.

\end{proof}

\section{Numerical results}\label{sec:num}
{
Let $\Omega = (0,1)^2$ and consider meshes $\mathfrak{T}=(\mesh,\edges,\centers,\vertices)$ as per \cite[Definition 7.2]{gdmbook}: $\mesh$ is the set of polygonal/polyhedral cells $K$, $\edges$ is the set of faces $\sigma$, $\centers$ is a set of points $(\x_K)_{K\in\mesh}$ with $K$ star-shaped with respect to $\x_K$ for all $K\in\mesh$, and $\vertices$ is the set of vertices $\vertex$.
Let us define two test cases.

\medskip
{\bf Case 1.}

This test case is divergence free (it corresponds to the pure transport of a tracer).
We choose $T = 5$, $u_{\rm ini}(\x) = 1$ if $\x = (x_1,x_2)\in (0.1, 0.4)\times(0.1, 0.4)$ and $u_{\rm ini}(x)=0$ elsewhere, $q^I = q^P = 0$ and $\velocity$ is given by 
\[
 \velocity(x_1,x_2) = ( (1-2 x_2)(x_1 - x_1^2), -(1-2 x_1)(x_2 - x_2^2)).
\]

{\bf Case 2.}

This test case includes source terms.
We choose $T=1$,  $u_{\rm ini}\equiv 0$,  $\velocity$ is given by 
\[
 \velocity(x_1,x_2) = ( x_1 - x_1^2, x_2 - x_2^2),
\]
$q^I(\x,t) = \max(\div( \velocity)(\x),0) =  \max(2-2(x_1+x_2),0)$ and $f(\x,t) = 1$, $q^P(\x,t) = \max(-\div( \velocity)(\x),0) = \max(2(x_1+x_2)-2,0)$. Then the solution $\bar u(\x,t)$ to \eqref{pbfort} can be analytically calculated. Find first $\X(s) = (X_1(s),X_2(s))$ by solving the differential equation $\X'(s) = \velocity(\X(s))$, letting $\X(0) = \widehat{\x} = (\widehat{x}_1,\widehat{x}_2)$ for any $\widehat{\x}\in\O$ (this is easy owing to the expression of $\velocity$); set then $\bar v(s) = \bar u(\X(s),s)$, which leads to $\bar v'(s) = (f(\X(s)) - \bar v(s)) q^I(\X(s))$ with $\bar v(\X(0)) = 0$. This requires to compute $\widehat{s}$ such that $X_1(\widehat{s})+X_2(\widehat{s})=1$, since $q^I(\X(s)) = 0$ for $s\ge \widehat{s}$. Finally, we get that $\bar u(\x,t) = \bar v(t)$, when $\widehat{\x}$ is chosen such that $\X(t) = \x$, and we have $\bar u(\x,t) = 0$ if $\widehat{x}_1+\widehat{x}_2>1$. Denoting by $\alpha := \min(1,e^{\widehat{s}-t})$, this leads to the following expressions.
\[
\begin{aligned}
& \alpha = \min\Big(1, \sqrt{\frac {(1-x_1)(1-x_2)} {x_1x_2} }\Big),\\
&\bar u(\x,t) =  1 - \Big( \frac {e^t(1+ x_1(\alpha-1))(1+ x_2(\alpha-1))}{\alpha(e^t(1-x_1)+x_1)(e^t(1-x_2)+x_2)}\Big)^2\hbox{ if }\alpha e^t \ge 1,~ \bar u(\x,t) = 0 \hbox{ otherwise}.
\end{aligned}
\]
}

\subsection{Case 1, different schemes with $p=2$}\label{ssec:num:case1}
We apply Scheme \eqref{eq:gradscheme} with three different gradient discretisations, corresponding respectively to the mass-lumped conforming $\mathbb{P}_1$ finite element method (or CVFE method, see \cite{forsyth} for the seminal paper and \cite[Chapter 8]{gdmbook} for the study in the GDM framework), to the mass-lumped non-conforming $\mathbb{P}_1$ (MLNC--$\mathbb{P}_1$ for short) finite element method \cite[Chapter 9]{gdmbook}, and to (a variant of) the Hybrid Finite Volume method (HFV), a member of the family of Hybrid Mimetic Mixed methods \cite[chapter 13]{gdmbook}. 
For the sake of completeness we briefly recall the definition of these gradient discretisations.

\begin{figure}[htb] 
 \begin{center}
\begin{tabular}{cc}
 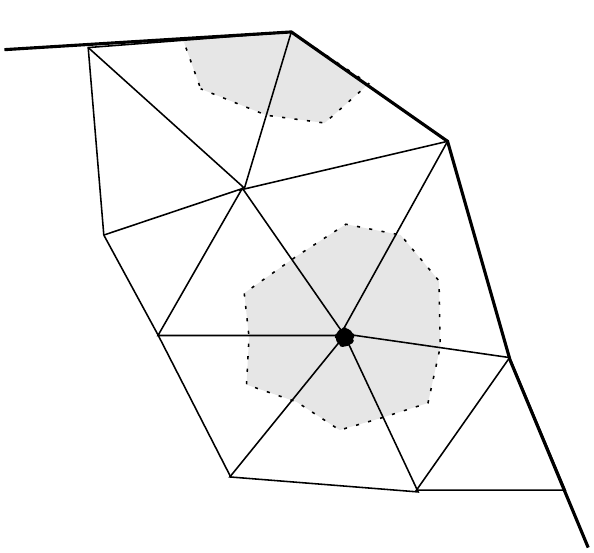 & \raisebox{3em}{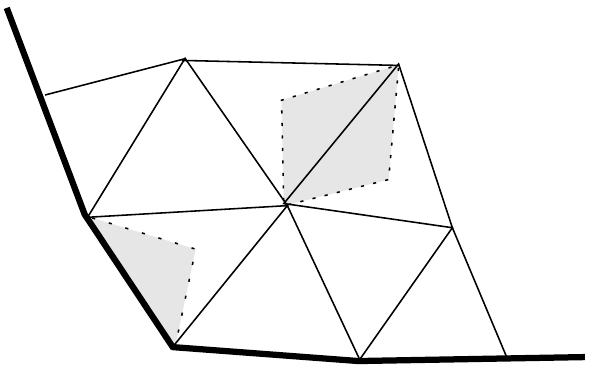}
\end{tabular}
 \caption{Dual cells for the mass-lumped $\mathbb{P}_1$ gradient discretisations: conforming $\mathbb{P}_1$ (CVFE, left) and non-conforming $\mathbb{P}_1$ (right).\label{fig:ml}}
 \end{center}
\end{figure}
\begin{itemize}
\item \emph{CVFE method (mass-lumped conforming $\mathbb{P}_1$)}: the mesh $\mathfrak{T}$ is a conforming simplicial mesh \cite[Definition 7.4]{gdmbook}, and
\begin{itemize}[$*$]
\item $X_\disc=\{u=(u_\vertex)_{\vertex\in\vertices}\,:\,u_\vertex\in\R\mbox{ for all }\vertex\in\vertices\}$.
\item For each vertex $\vertex\in\vertices$, a dual cell $C_\vertex$ is constructed around the vertex by joining the cell centres of mass, the face centres of mass and (in 3D) the edge midpoints around $\vertex$ (see Figure \ref{fig:ml}, left).
Then, for $u\in X_\disc$ and any $\vertex\in\vertices$, $(\Pi_\disc u)_{|C_\vertex}=u_\vertex$.
\item For $u\in X_\disc$, $\nabla_\disc u$ is the gradient of the $\mathbb{P}_1$ function constructed from the vertex values $(u_\vertex)_{\vertex\in\vertices}$.
\end{itemize}
{The proof that this GD leads to a consistent and limit-conforming sequence of space-time gradient discretisation in the sense of Definition \ref{def:conslimconf} can be obtained by following that of  \cite[Theorem 8.17]{gdmbook}.}

\item \emph{MLNC--$\mathbb{P}_1$}: the mesh $\mathfrak{T}$ is also a conforming simplicial mesh, and
\begin{itemize}[$*$]
\item $X_\disc=\{u=(u_\edge)_{\edge\in\edges}\,:\,u_\edge\in\R\mbox{ for all }\edge\in\edges\}$.
\item For each face $\edge\in\edges$, a dual cell $C_\edge$ is constructed as the union, for each cell on each side of $\edge$, of the convex hulls on the face and the cell centre of mass (see Figure \ref{fig:ml}, right).
Then, for $u\in X_\disc$ and any $\edge\in\edges$, $(\Pi_\disc u)_{|C_\edge}=u_\edge$.
\item For $u\in X_\disc$, $\nabla_\disc u$ is the gradient of the non-conforming $\mathbb{P}_1$ function constructed from the edge values $(u_\edge)_{\edge\in\edges}$.
\end{itemize}
{The proof that this GD leads to a consistent and limit-conforming sequence of space-time gradient discretisation in the sense of Definition \ref{def:conslimconf} can be obtained by following that of  \cite[Theorem 9.17]{gdmbook}.}

\item \emph{(Variant of the) HFV method}: $\mathfrak{T}$ is a generic polygonal/polyhedral mesh and
\begin{itemize}[$*$]
\item $X_\disc=\{u=((u_K)_{K\in\mesh},(u_\sigma)_{\sigma\in\edges}\,:\,u_K\in\R\mbox{ for all $K\in\mesh$},\ u_\edge\in\R\mbox{ for all $\edge\in\edges$}\}$.
\item A coefficient  $\gamma\in (0,1]$ is chosen and each cell $K$ is partitioned into $\hat{K}$ and $(K_\edge)_{\edge\in\edges_K}$, where $\edges_K$ is the set of faces of $K$, $|\hat{K}|=\gamma|K|$ and $|K_\edge|=\frac{1-\gamma}{{\rm Card}(\edges_K)}|K|$ (here, $|E|$ denotes the Lebesgue measure of the set $E$).
\item For all $u\in X_\disc$ and all $K\in\mesh$, $(\Pi_\disc u)_{|\hat{K}}=u_K$ and, for all $\edge\in\edges_K$, $(\Pi_\disc u)_{|K_\edge}=u_\edge$.
\item For all $u\in X_\disc$, all $K\in\mesh$ and all $\sigma\in\edgescv$ (where $\edgescv$ is the set of faces of $K$),
\[
(\nabla_\disc u)_{|D_{K,\edge}}=\overline{\nabla}_K u + \beta_K\frac{\sqrt{d}}{d_{K,\edge}}\left[u_\edge-u_K-\overline{\nabla}_K u\cdot (\overline{\x}_\edge-\x_K)\right]\bfn_{K,\edge},
\]
where $\beta_K>0$ is a user-defined parameter and
\begin{itemize}[$\diamond$]
\item $\bfn_{K,\edge}$ and $\overline{\x}_\edge$ are respectively the outer normal to $K$ on $\edge$ and the centre of mass of $\edge$,
\item $\overline{\nabla}_K u=\frac{1}{|K|}\sum_{\edge\in\edgescv} |\edge|u_\edge\bfn_{K,\edge}$, with $|K|$ and $|\edge|$ the $d$- and $(d-1)$-measure of $K$ and $\edge$, respectively,
\item $d_{K,\edge}$ the orthogonal distance between $\x_K$ and $\edge$.
\end{itemize}
\end{itemize}
{The proof that this GD leads to a consistent and limit-conforming sequence of space-time gradient discretisation in the sense of Definition \ref{def:conslimconf} can be obtained by following that of \cite[Theorem 13.16]{gdmbook}.}

\end{itemize}

\begin{remark}[Original HFV method]\label{rem:sushi}
The original HFV scheme (also known as SUSHI scheme) consists in choosing $\gamma=1$ \cite[chapter 13]{gdmbook}, that is, $(\Pi_\disc u)_{|K}=u_K$ for all $K\in\mesh$ (the face unknowns are not involved in the definition of $\Pi_\disc$). 
We however found that, when applied to the gradient scheme \eqref{eq:gradscheme} for the linear hyperbolic equation, the HFV method requires quite a lot of fiddling with various parameters (diffusion magnitude and direction, the coefficients $\beta_K$, etc.) to produce acceptable results. 
 Indeed, for $\gamma=1$, the face unknowns are not involved in the accumulation term in \eqref{eq:gradscheme}, so that these unknowns are not accurately updated at each time step -- the diffusion is the quantity that links the face and cell unknowns, and with a vanishing diffusion, this link looses too much strength. 
Involving the face unknowns in the definition of $\Pi_\disc$, by re-distributing the fraction $1-\gamma$ of the complete volume to these unknowns in the accumulation term, ensures a much better stability and behaviour of the method.
However, for $\gamma=0$, {\it  i.e.} when  the total volume $|K|$ is re-distributed so that only the face unknowns are accounted for in $\PiD$, the solution displays severe oscillations around the discontinuities of the initial condition. 
The coefficient $\gamma$ should therefore be chosen in $(0,1)$.
\end{remark}
\begin{remark}[Choice of {$\hat{K}$} and {$K_\edge$}]
In practice, implementing the HFV method does not require to choose a detailed geometry for $\hat{K}$ and $K_\edge$, as source and advection integral terms are approximated using only the values of the function at the centres of mass of $K$ and $\sigma$  and the measures of $\hat{K}$ and $K_\sigma$.
For example,
\[
\int_\Omega f \Pi_\disc v \d\x \approx \sum_{K\in\mesh} \left(|\hat{K}|f(\overline{\x}_K)v_K +\sum_{\edge\in\edges_K}|K_\edge|f(\overline{\x}_\edge)v_\edge\right),
\]
where, for $E=K$ or $E=\edge$, $\overline{\x}_E$ is the centre of mass of $E$.
\end{remark}

We also compare the results, obtained with these GDs, with the results using the upstream weighting scheme based on the standard CVFE method \cite[Section 4.3]{chen2006comp} on a triangular mesh (upstream values are computed with respect to the sign of fluxes computed at the boundaries of the dual mesh). 
All the considered meshes are from \cite{2Dbench}. 
For the CVFE, MLNC--$\mathbb{P}_1$ and upstream schemes we use the family of meshes \texttt{mesh1\_X}. 
For the HFV method we fixed $\gamma=0.3$, $\beta_K=1$ for all $K\in\mesh$ and we ran the simulations on the locally refined and non-conforming family of meshes \texttt{mesh3\_X}. 
A sensitivity analysis on the parameter $\gamma$ was carried out. 
Tests were performed for $\gamma$ ranging from 0 to 1.
As mentioned in Remark \ref{rem:sushi} for $\gamma=0$, severe oscillations occur because the cell unknowns are no longer present in the accumulation term.
The numerical results obtained for $\gamma \in (0,1)$ do not vary much, although taking $\gamma\in (0,1)$ instead of $\gamma=1$ seems to reduce the numerical diffusion and produces a scheme which is more stable with respect to changes in the parameter $\Lambda$.
Examples of the considered mesh families are shown in Figure \ref{fig:meshes}.
We let $\theta =\half$, $p=2$ and $\alpha = 2$ for the discretisation scheme (note that we only proved that the scheme converges for $\alpha\in(0,2)$). 
The analytical solution is approximated by the characteristics method, where the characteristics ODE is approximated using the explicit Euler scheme with time step $0.001$.

The errors are calculated at the final time, by projecting the analytical solution onto the appropriate piecewise-constant functions (depending on the considered method). Thus, for $q=1$ or $q=2$, we set
\[
\begin{aligned}
\ba\mbox{CVFE and}\\\mbox{upstream $\mathbb{P}_1$}\ea:&\quad \texttt{errlq}=\left(\sum_{\vertex\in\vertices} |C_\vertex|\,|u^N_\vertex-\bar u(\vertex,T)|^q\right)^{1/q},\\
\mbox{MLNC--$\mathbb{P}_1$}:&\quad \texttt{errlq}=\left(\sum_{\edge\in\edges} |C_\edge|\,|u^N_\edge-\bar u(\overline{\x}_\edge,T)|^q\right)^{1/q},\\
\mbox{HFV}:&\quad \texttt{errlq}=\left(\sum_{K\in\mesh} \gamma|K|\,|u^N_K-\bar u(\overline{\x}_K,T)|^p
+\sum_{K\in\mesh}\sum_{\edge\in\edges_K}\frac{1-\gamma}{{\rm Card}(\edges_K)}|K| \,|u^N_\edge-\bar u(\overline{\x}_\edge,T)|^q\right)^{1/q}.
\end{aligned}
\]

\begin{figure} 
 \begin{center}
\begin{tabular}{cc}
 \includegraphics[width=0.9\linewidth]{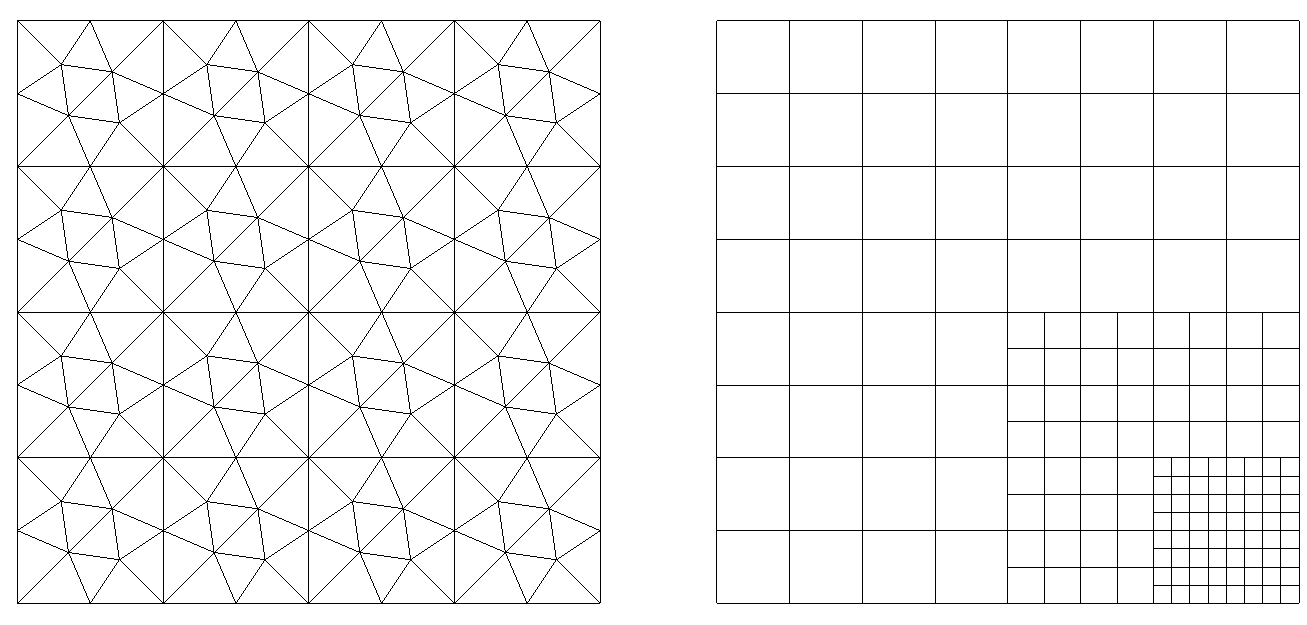}
\end{tabular}
 \caption{Meshes used for the simulations: \texttt{mesh1\_2} (left) and \texttt{mesh3\_2} (right).\label{fig:meshes}}
 \end{center}
\end{figure}

\begin{table}[ht!]
\begin{center}
\begin{tabular}{|c|c|c|c|c|c|c|c|}
\hline 
$h$ & \texttt{errl2} & rate & \texttt{errl1} & rate & umin& umax\\
\hline
 0.250 &  2.95E-01 &    -   &  1.95E-01 &    -   &    0.108 &  0.137 \\
\hline 
 0.125 &  2.55E-01 &  0.212 &  1.37E-01 &  0.504 &    0.014 &  0.174 \\
\hline 
 0.062 &  2.32E-01 &  0.136 &  1.23E-01 &  0.158 &    0.000 &  0.344 \\
\hline 
 0.031 &  1.77E-01 &  0.394 &  8.55E-02 &  0.525 &   -0.001 &  0.734 \\
\hline 
 0.016 &  1.23E-01 &  0.524 &  4.73E-02 &  0.853 &   -0.013 &  1.003 \\
\hline
\end{tabular}
\end{center}
\caption{\label{tab:cvfe-cent} {Case 1}, results with the centred scheme, using the CVFE method, $\dt = 0.4 h$ }
\end{table}

\begin{table}[ht!]
\begin{center}
\begin{tabular}{|c|c|c|c|c|c|c|c|}
\hline 
$h$ & \texttt{errl2} & rate & \texttt{errl1} & rate & umin& umax\\
\hline
 0.250 & 2.52E-1  &    -   &  1.10E-1  &    -   &   0.043   &  0.054  \\
\hline
 0.125 & 2.65E-1  & -0.076  &  1.51E-1  &  -0.457  &  0.016    &  0.194  \\
\hline
 0.062 & 2.37E-1 & 0.165  & 1.31E-1  & 0.208  &   0.000   & 0.361  \\
\hline
 0.031 & 1.82E-1  & 0.381  & 8.64E-2   & 0.597 &  0.000   &  0.687 \\
\hline
 0.016 & 1.33E-1  & 0.456   & 5.34E-2   & 0.694   & 0.000   & 0.960   \\
\hline
\end{tabular}
\end{center}
\caption{\label{tab:mnncP1-cent} {Case 1}, results with the centred scheme, using the MLNC--$\mathbb{P}_1$ method, $\dt = 0.4 h$ }
\end{table}

\begin{table}[ht!]
\begin{center}
\begin{tabular}{|c|c|c|c|c|c|c|c|}
\hline 
h & \texttt{errl2} & rate & \texttt{errl1} & rate & umin& umax\\
\hline
0.35 & 2.80E-1  &    -   &  2.08E-1  &    -   &  0.152    & 0.155   \\
\hline
0.18 & 2.79E-1  & 0.001  &  1.54E-1  & 0.436   & 0.044     &  0.124  \\
\hline
0.09 & 2.59E-1 & 0.111  & 1.30E-1  & 0.236  &  0.001    &  0.220 \\
\hline
0.04 & 2.10E-1  & 0.300  & 1.08E-1   & 0.276 &  0.000   & 0.499  \\
\hline
0.02 & 1.47E-1  &  0.520  & 6.57E-2   & 0.713   &  0.000  & 0.906   \\
\hline
\end{tabular}
\end{center}
\caption{\label{tab:HMM-cent} {Case 1}, results with the centred scheme, using the HFV method, $\dt = 0.4 h$ }
\end{table}

\begin{table}[ht!]
\begin{center}
\begin{tabular}{|c|c|c|c|c|c|c|c|}
\hline 
$h$ & \texttt{errl2} & rate & \texttt{errl1} & rate & umin& umax\\
\hline
 0.250 &  2.59E-01 &    -   &  1.65E-01 &    -   &    0.005 &  0.313 \\
\hline
 0.125 &  2.32E-01 &  0.159 &  1.19E-01 &  0.462 &    0.000 &  0.286 \\
\hline
 0.062 &  2.13E-01 &  0.122 &  1.10E-01 &  0.122 &    0.000 &  0.454 \\
\hline
 0.031 &  1.85E-01 &  0.205 &  9.13E-02 &  0.266 &    0.000 &  0.672 \\
\hline
 0.016 &  1.53E-01 &  0.270 &  6.93E-02 &  0.398 &    0.000 &  0.868 \\
\hline
\end{tabular}
\end{center}
\caption{\label{tab:cvfe-ups} {Case 1}, results with the upstream $\mathbb{P}_1$ scheme, $\dt = 0.4 h$ }
\end{table}

We observe that all the convergence rates are lower than one half (due to the discontinuity of the exact solution, better orders cannot be expected). The GDM based methods seem to produce such an order when refining the meshes. {Note that these convergence orders are much smaller than that observed on Test Case 2 (see Section \ref{sec:resnumcase2}), which can be expected since the analytical solution is discontinuous here, whereas it is continuous in Test Case 2.}

\newcommand{\inc}[1]{\includegraphics[viewport = 600 150 1100 650,clip = true]{#1}}

\begin{figure}[ht!]
\begin{center}

\setlength\unitlength{1cm}

\resizebox{\textwidth}{!}{\inc{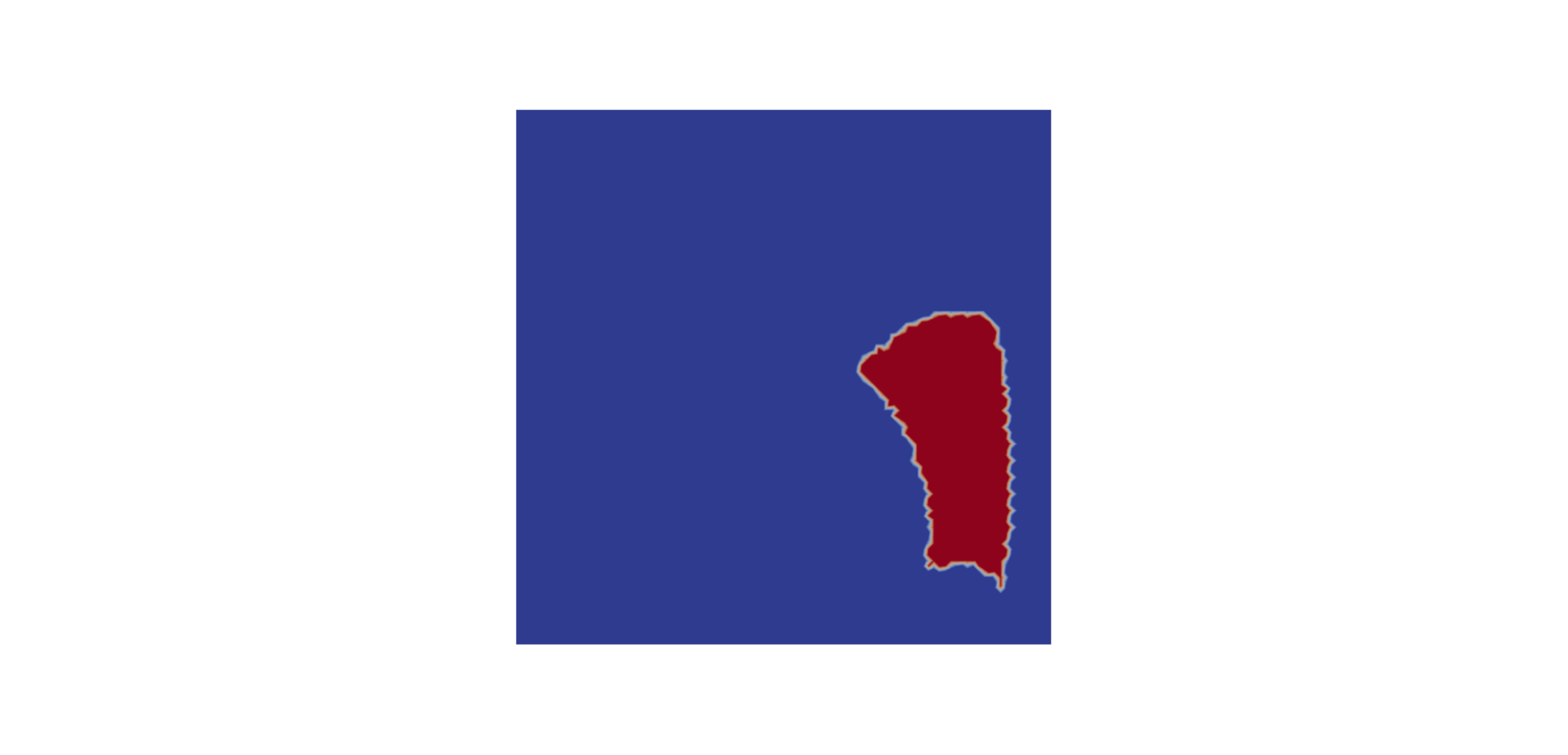}\inc{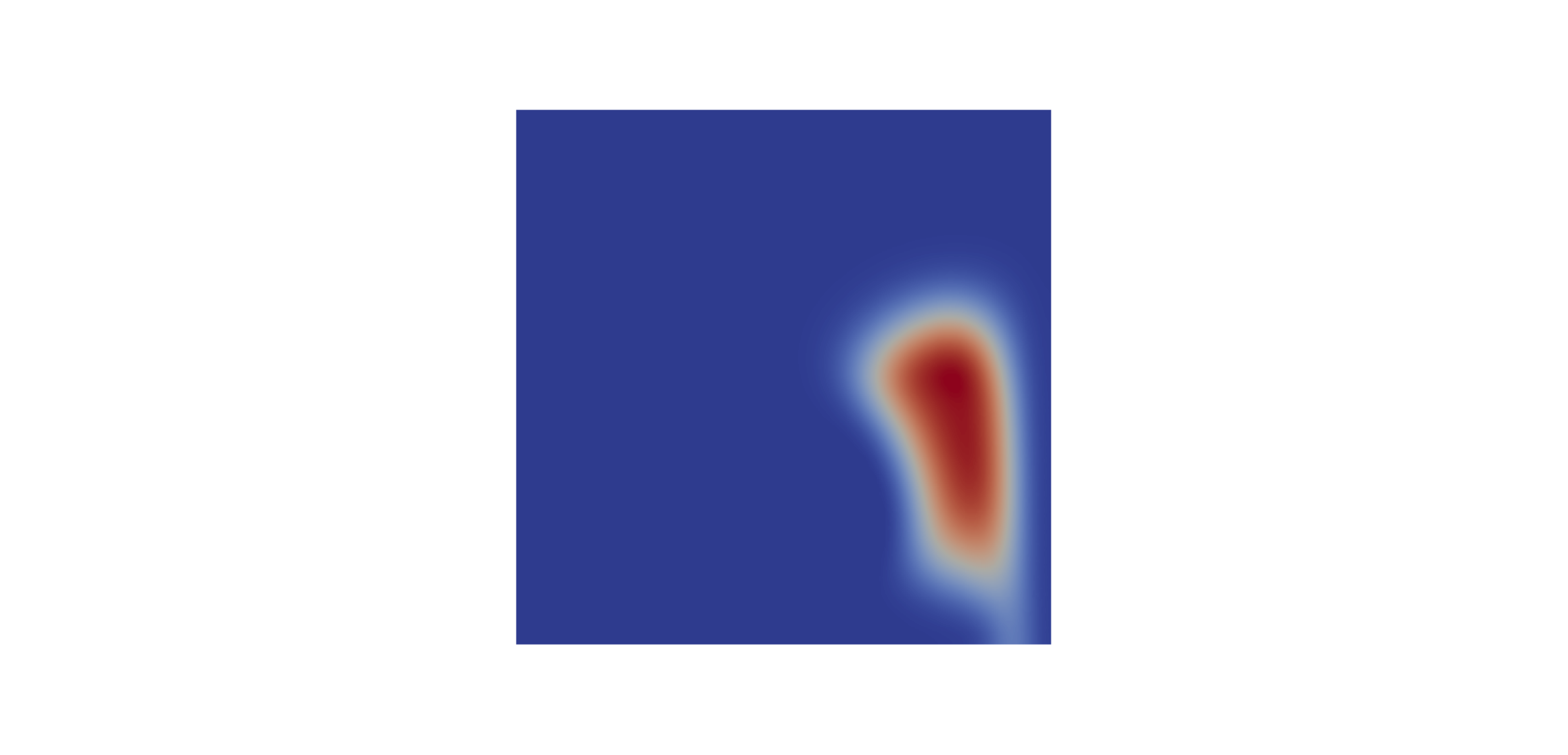}\inc{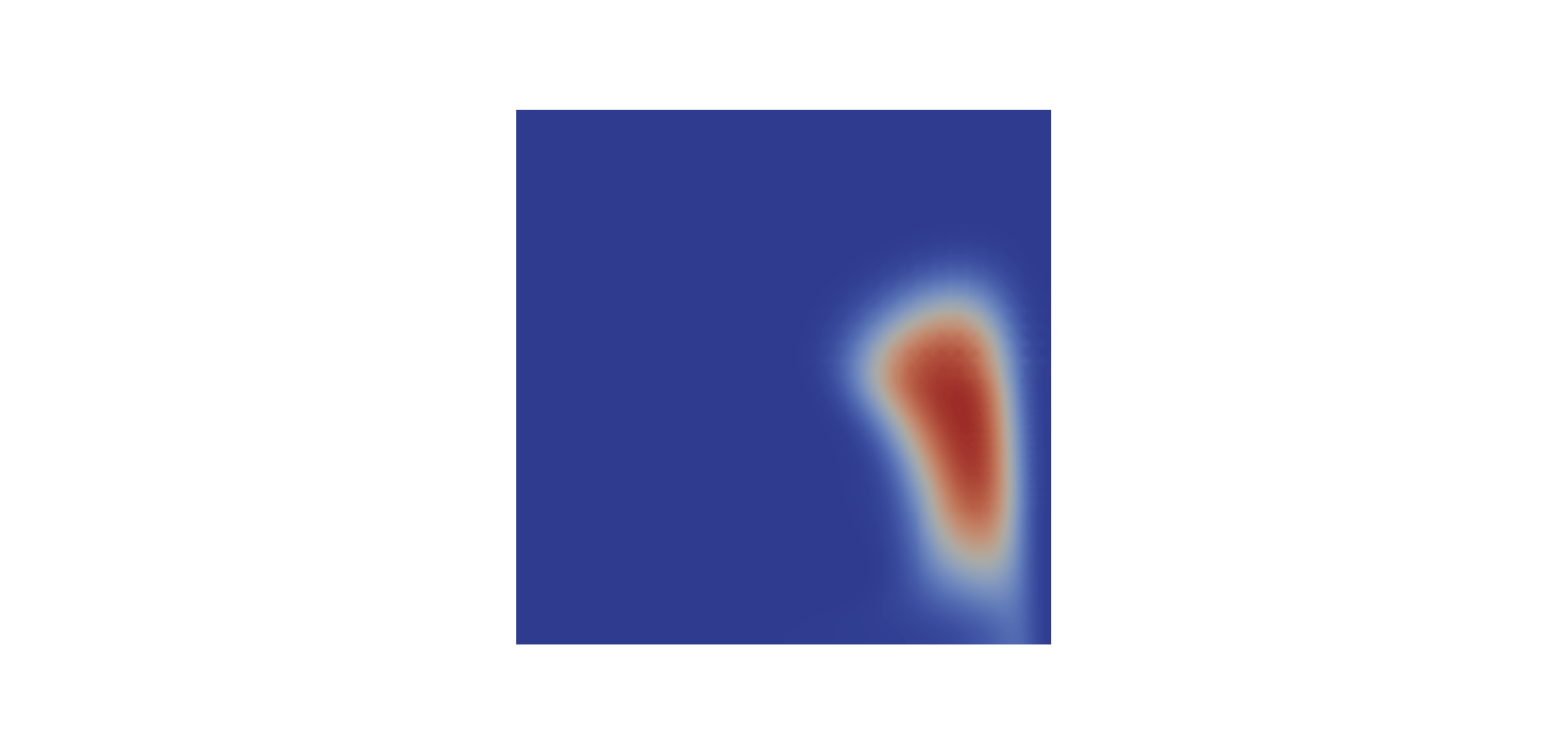}\inc{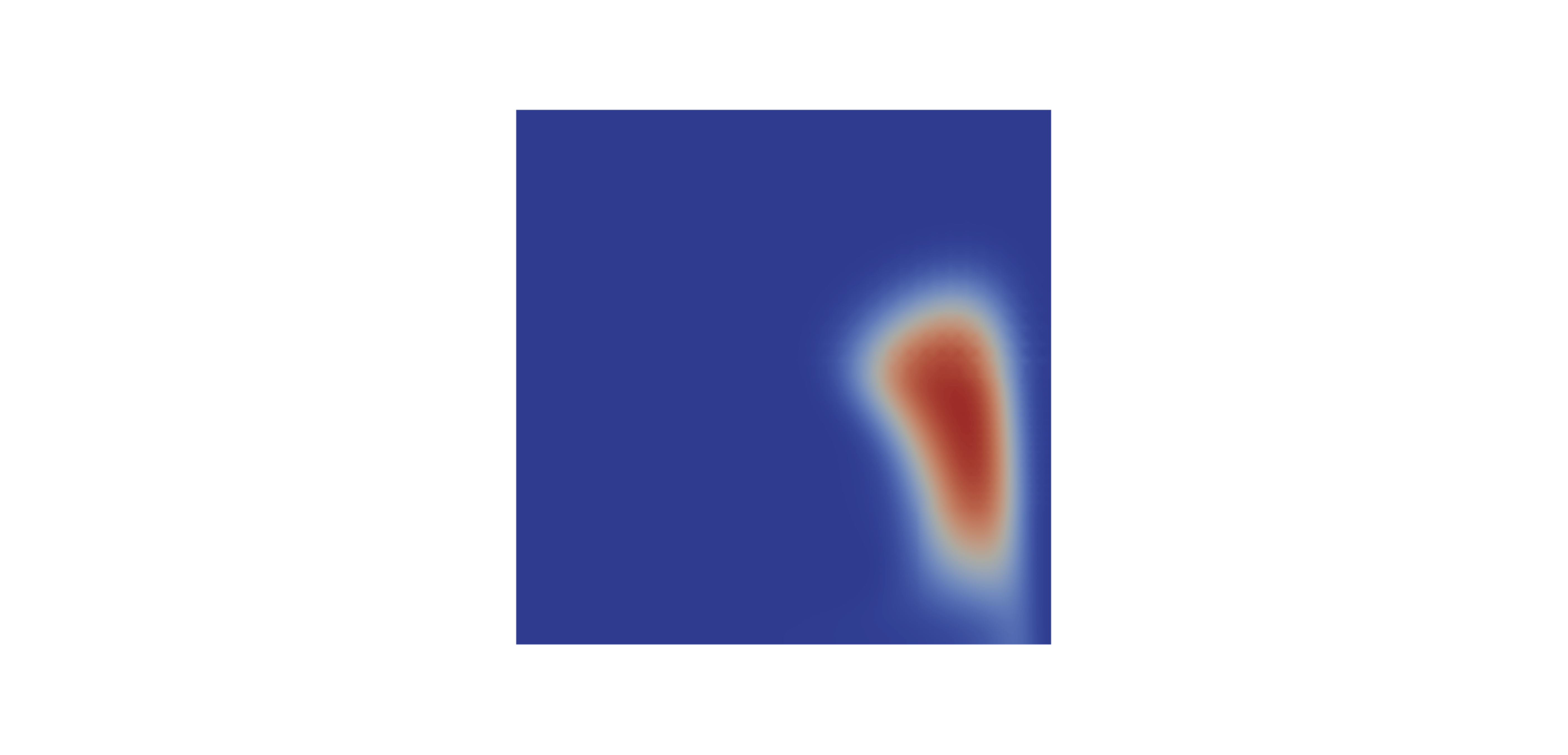}\inc{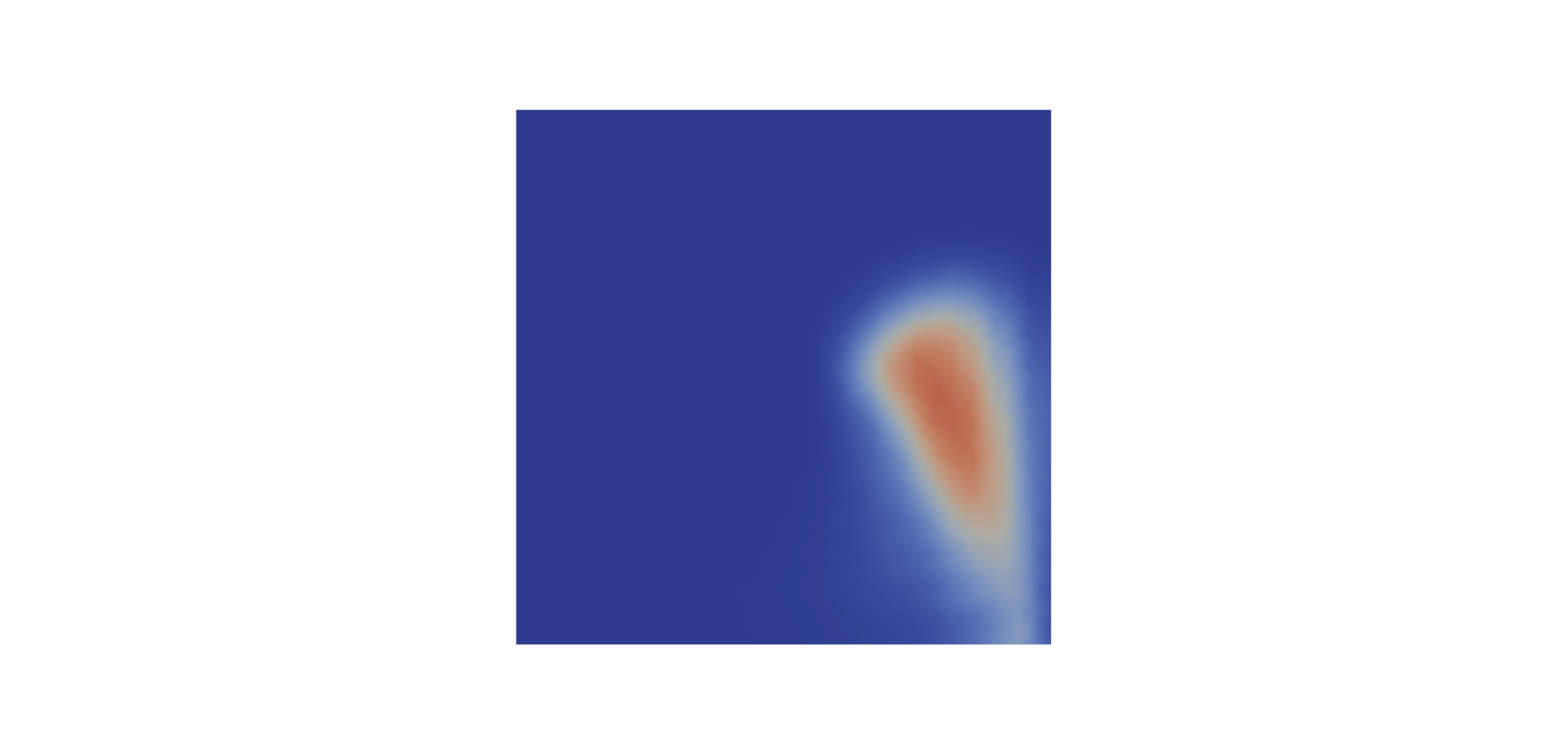}}
\begin{picture}(5,7)
\put(0.1,0.1){\resizebox{4.5cm}{!}{\includegraphics{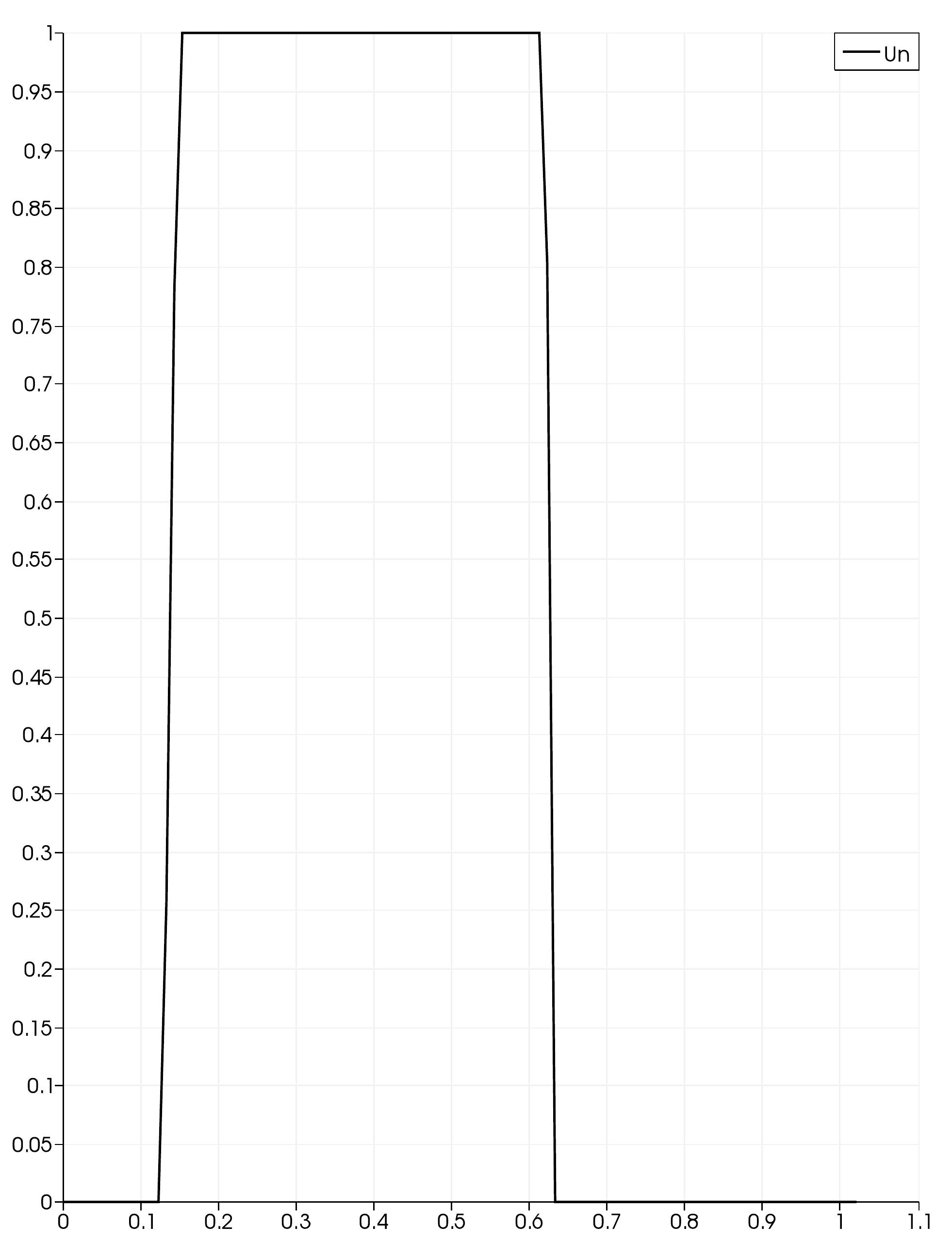}}}
\put(4.7,0){$s$}
\put(0,6){$\bar u(\x(s),T)$}
\end{picture}
\begin{picture}(5,7)
\put(0.1,0.1){\resizebox{4.5cm}{!}{\includegraphics{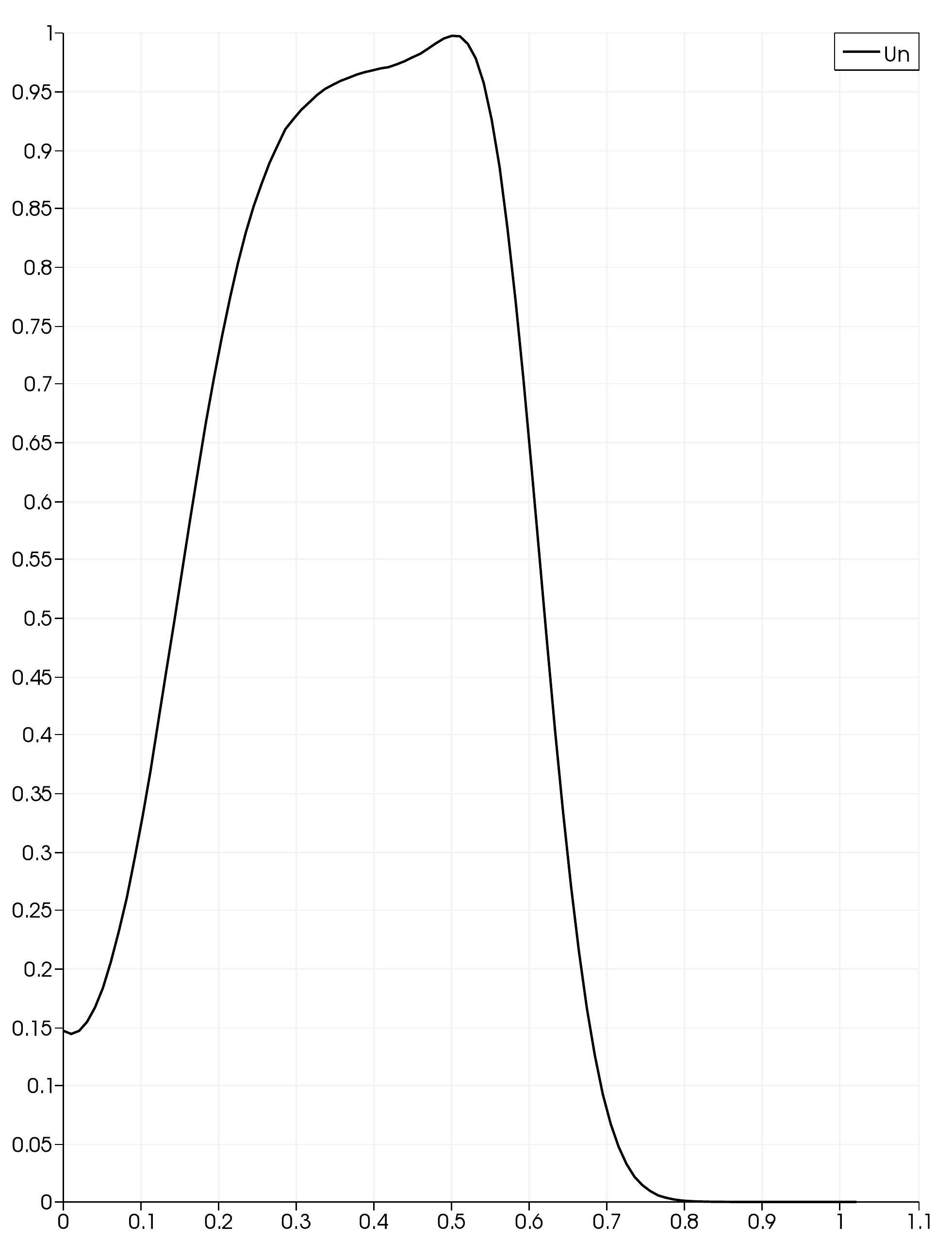}}}
\put(4.7,0){$s$}
\put(0,6){$\Pi_\disc u^{(N)}(\x(s))$}
\end{picture}
\begin{picture}(5,7)
\put(0.1,0.1){\resizebox{4.5cm}{!}{\includegraphics{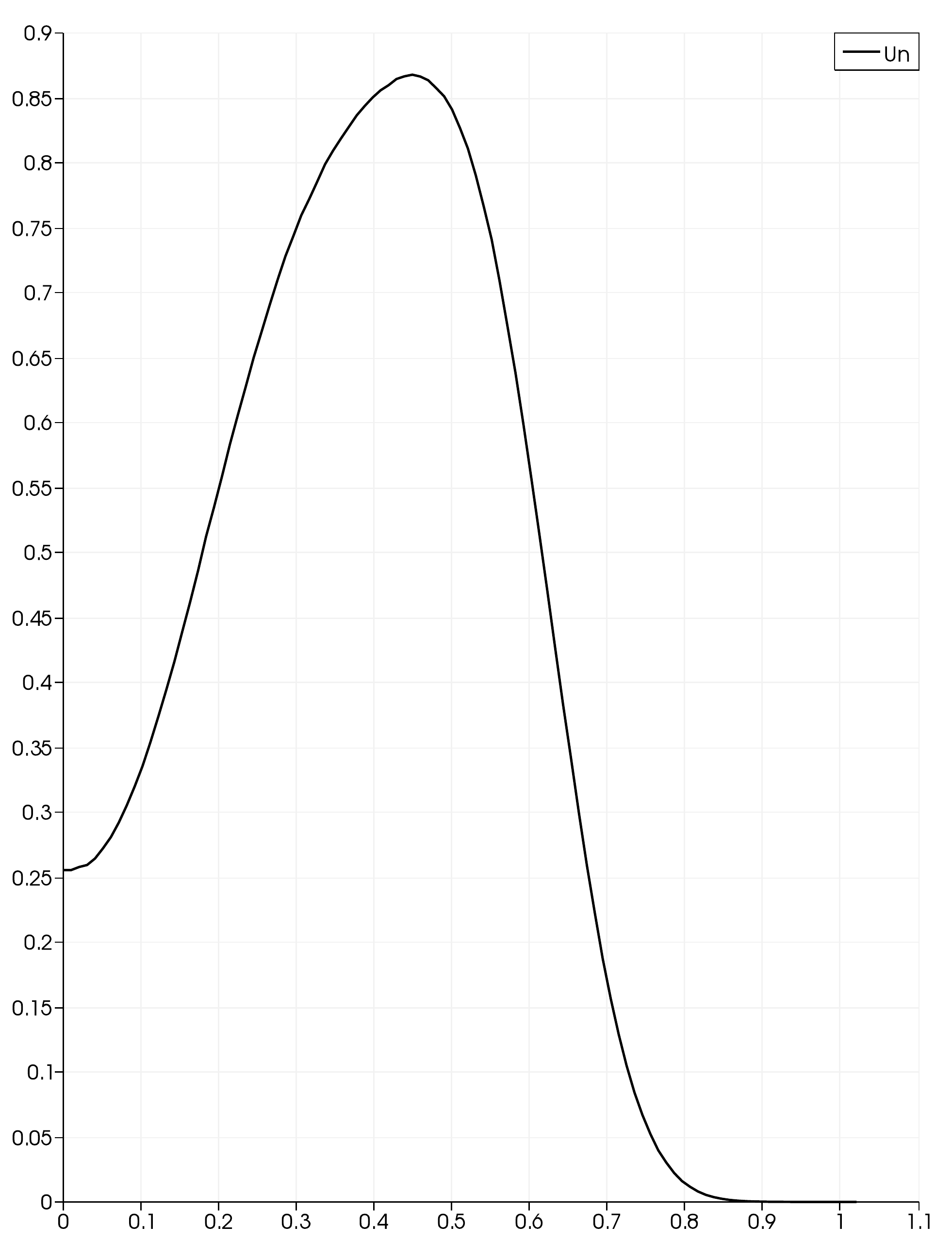}}}
\put(4.7,0){$s$}
\put(0,6){$\Pi_\disc u^{(N)}(\x(s))$}
\end{picture}
\end{center}
\caption{\label{fig:profiles_2d} Top: values of $u$ (darkblue 0, darkred 1). Left to right: analytical solution at final time $T=5$, centred CVFE, upwind CVFE, HFV, MLNC--$\mathbb{P}_1$ at last step $N$. 
Bottom: Profiles at the same times and time step along the segment $[\x(0),\x(\sqrt{1.04})]$  with curvilinear abscissae $s$ such that $\x(0) = (0.9,0)$ and $\x(\sqrt{1.04}) = (0.7,1)$, with $\theta = 0.5$. Left: analytical solution; centre: centred scheme with $\alpha = 2$; right: upstream $\mathbb{P}_1$ scheme.}
\end{figure}
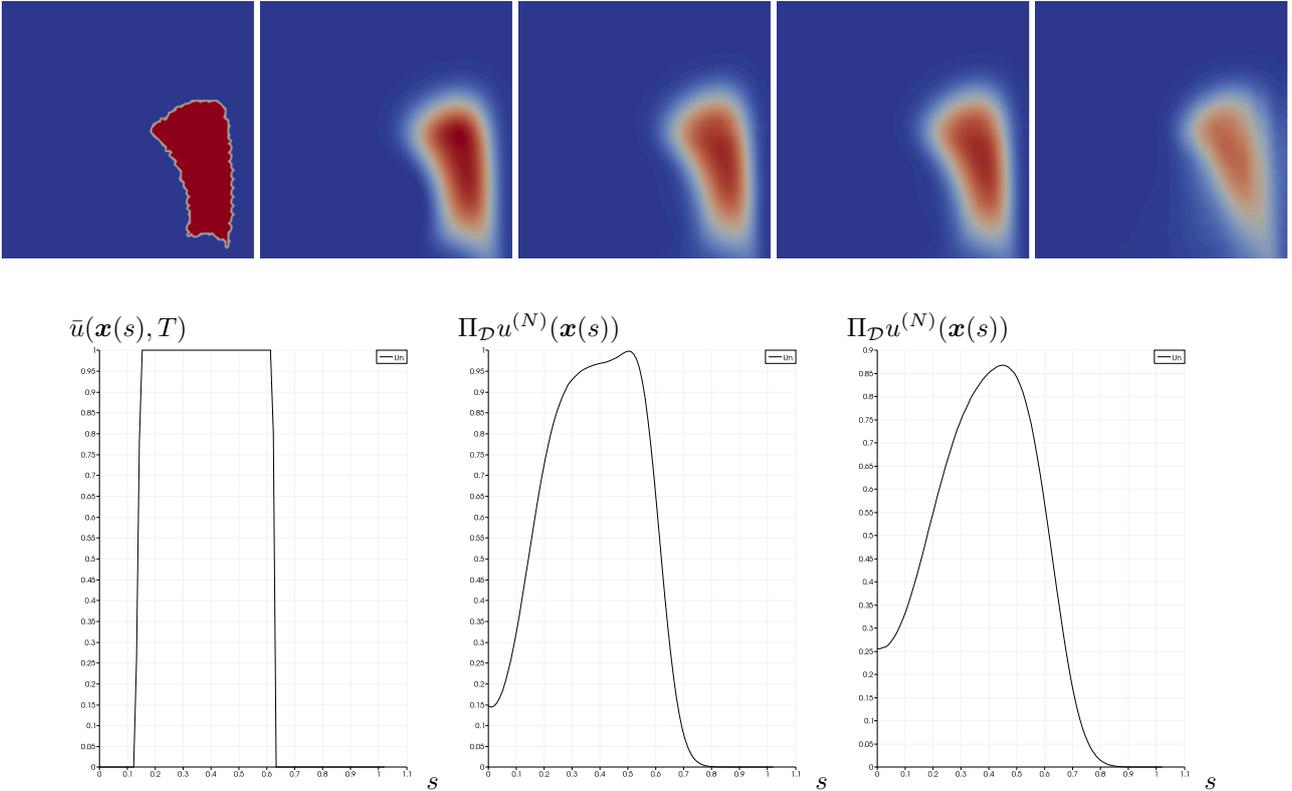

\subsection{Case 1, different values of $p$ with CVFE scheme}

In Figure \ref{fig:profiles_p}, we compare on the triangular mesh \texttt{mesh1\_4} the results obtained on the same problem as the previous section, but using only the CVFE scheme, and letting $p$ vary. 
The numerical scheme is solved quite accurately at each time step using Newton's method (the $p$-Laplace operator being particularly easy to compute using the $\mathbb{P}^1$ finite element). The homogeneity degree of the coefficient for the diffusion term, with respect to the units of length and $\bar u$, is a function of $p$. Because of that, properly comparing the results for various $p$ is difficult at best.

The considered mesh \texttt{mesh1\_4} is too coarse for the scheme to have already converged. However, there is something to be learnt on the results on this mesh since computing numerical solutions on a too coarse mesh is a standard situation in industrial contexts. We observe that, on this mesh, the profiles obtained with $p<2$ differ quite a bit from those obtained with $p\ge 2$, the latter being closer to the expected solution. This seems to indicate that, in practical applications, choosing a higher value of $p$ provides better results.

\begin{figure}[ht!]
\begin{center}
\setlength\unitlength{1cm}
\begin{picture}(15,9)
\put(0.1,0.1){\resizebox{15cm}{!}{\includegraphics{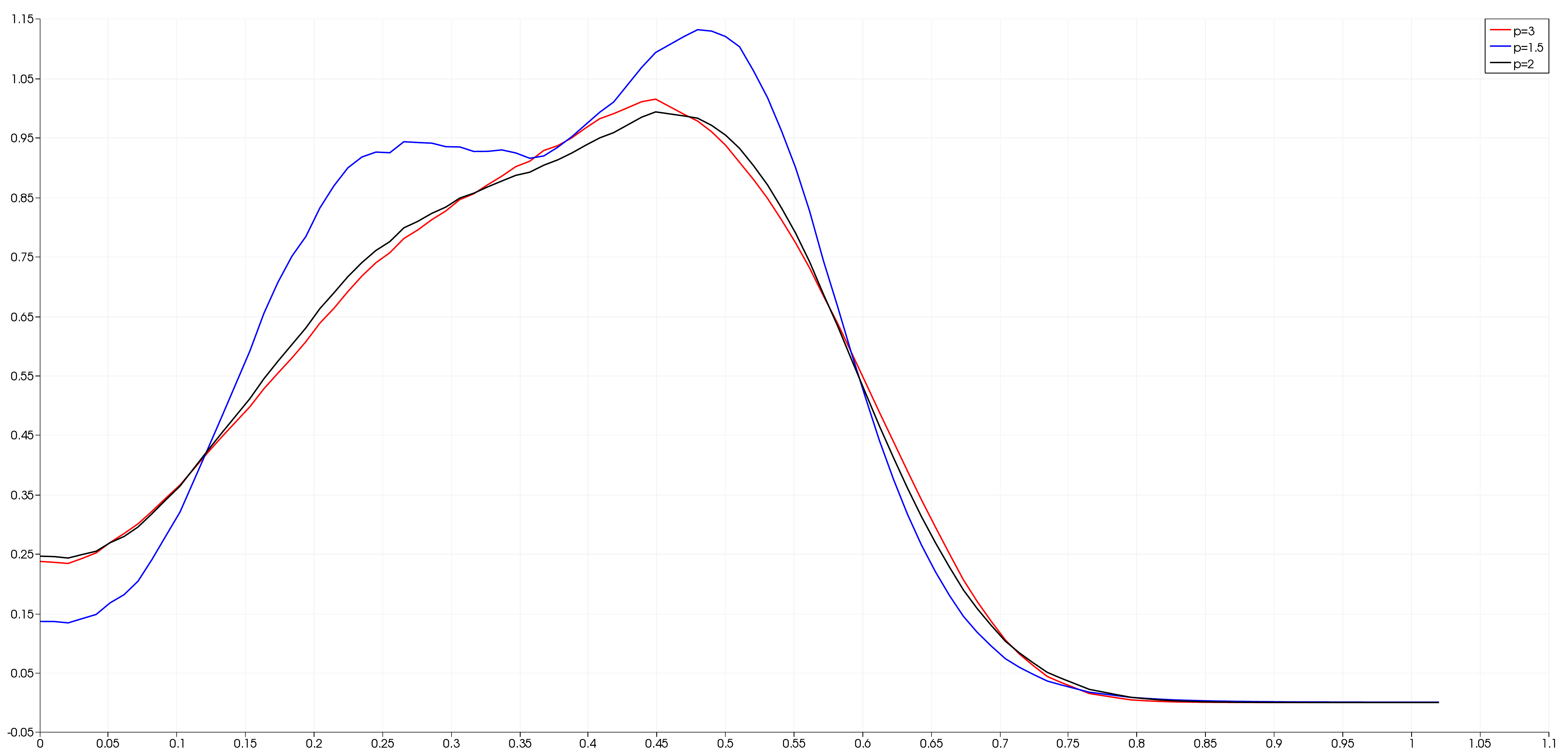}}}
\put(15.2,0){$s$}
\put(0,7.5){$\Pi_\disc u^{(N)}(\x(s))$}
\end{picture}
\end{center}
\caption{\label{fig:profiles_p} Comparison of the solutions $\Pi_\disc u^{(N)}(\x(s))$ at time step $N$ for different values of $p$ ($p=3$ for the red curve, $p=1,5$ for the blue curve and $p=2$ for the black curve), along the same segment as in Figure \ref{fig:profiles_2d}.}
\end{figure}

\subsection{Case 2, comparison of CVFE schemes with analytical solution}\label{sec:resnumcase2}
{
We apply Scheme \eqref{eq:gradscheme} with the CVFE method to Case 2, with $p=2$ and $\theta = 0.5$. In this case, the solution $\bar u$ is regular (it belongs to $C^1(\O\times[0,T])$), and therefore the convergence orders are much higher that those observed in Test Case 1; this is shown in Table \ref{tab:cvfe-cent-div}, which provides the convergence orders including for the $L^\infty$-norm at the final time. The convergence orders with Scheme  \eqref{eq:gradscheme} are also higher that the ones observed in Table \ref{tab:cvfe-ups-div} for the upstream scheme (these orders for the $L^1$ error are close to 1, as expected in this regular case).
This accurate convergence is confirmed by Figure \ref{fig:soldiv}, where we plot the profiles of the approximate solutions and of the exact solution at final time along the first diagonal. }

\begin{figure}[ht!]
\begin{center}
\resizebox{4cm}{!}{\includegraphics[trim={4cm 3cm 4cm 2cm},clip]{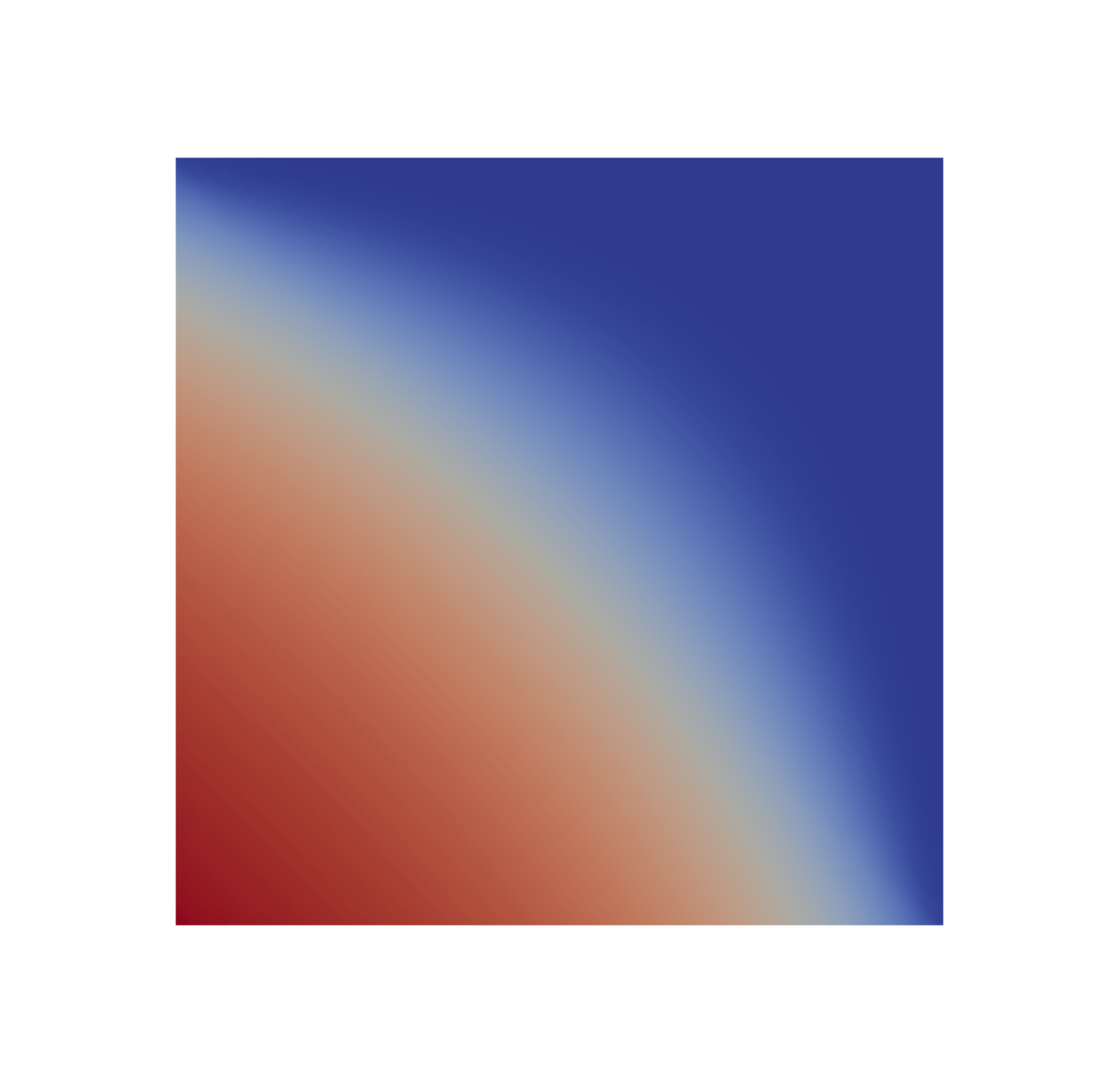}}
\setlength\unitlength{1cm}
\begin{picture}(5,5)
\put(0.1,-1.8){\resizebox{6cm}{!}{\includegraphics{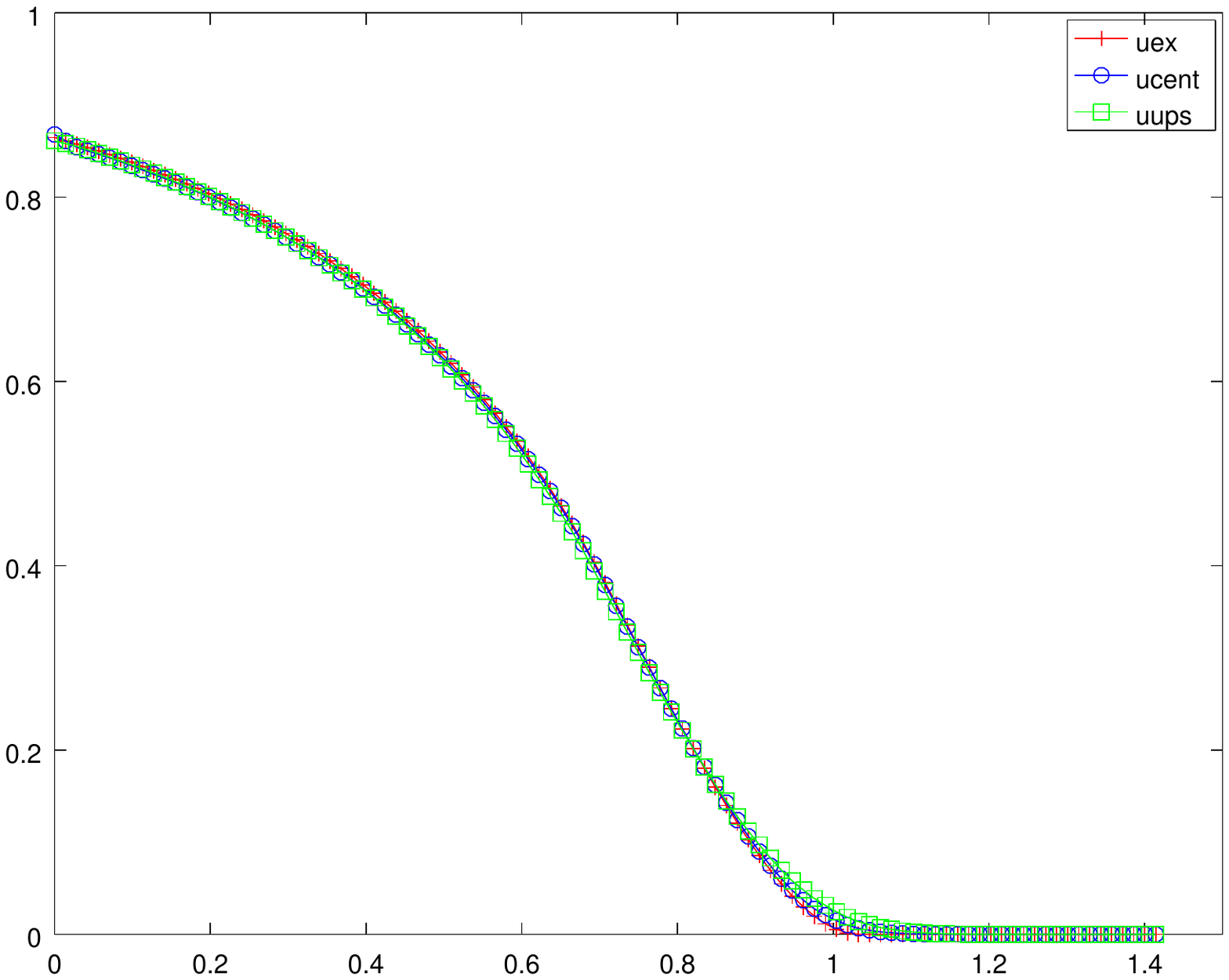}}}
\put(5.5,0){$s$}
\put(0.8,4.5){\begin{minipage}{5em}\small $\bar u(\x(s),T)$\\ and $\Pi_\disc u^{(N)}(\x(s))$\end{minipage}}
\end{picture}
\setlength\unitlength{1cm}
\begin{picture}(5,4)
\put(0.1,-1.8){\resizebox{6cm}{!}{\includegraphics{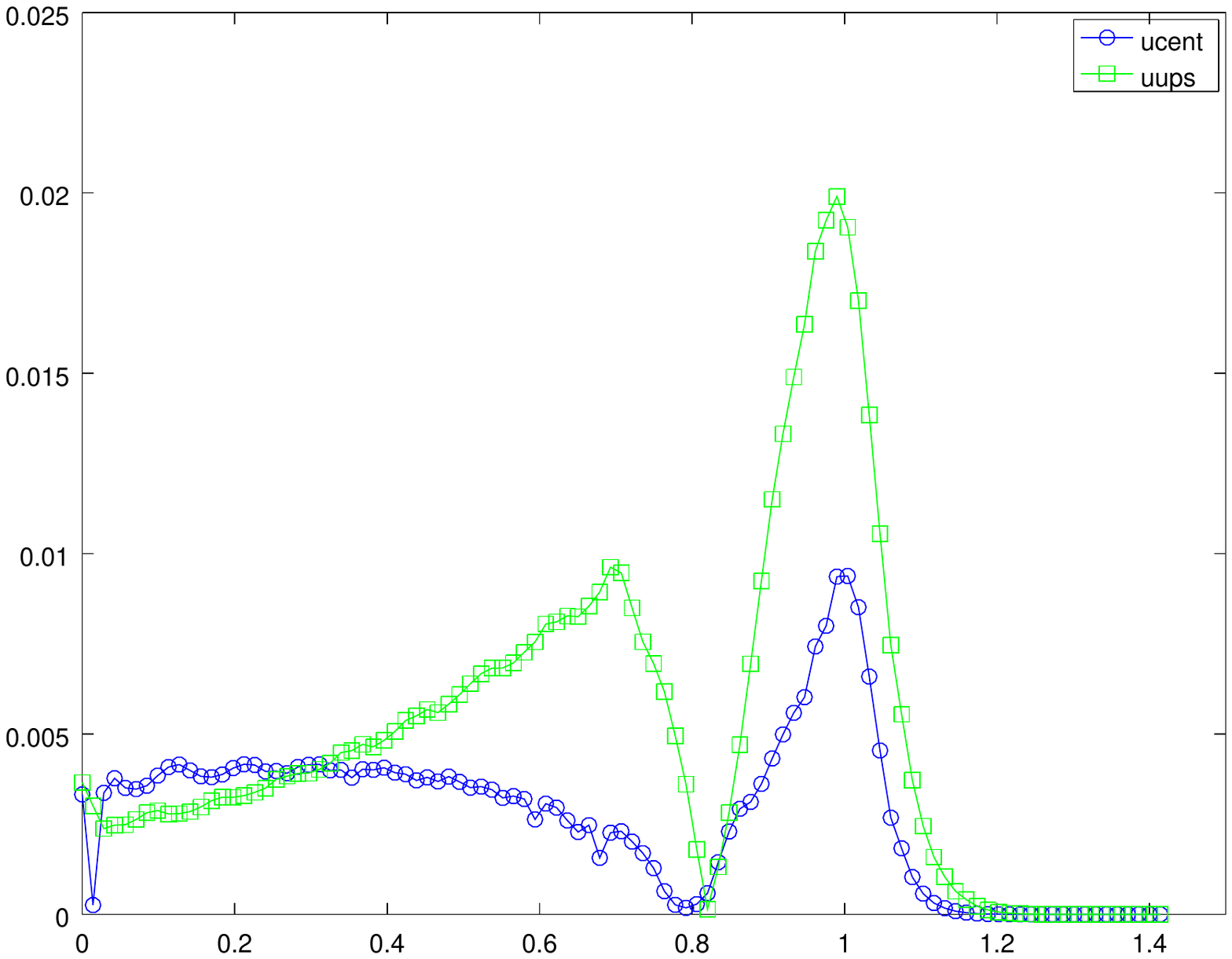}}}
\put(5.5,0){$s$}
\put(0.8,4.3){\begin{minipage}{15em}\small $|\Pi_\disc u^{(N)}(\x(s))-\bar u(\x(s),T)|$\end{minipage}}
\end{picture}
\end{center}
\caption{\label{fig:soldiv} {Case 2. Left: approximate solution at the last time step with  Scheme \eqref{eq:gradscheme}. Middle: $\Pi_\disc u^{(N)}(\x(s))$ and exact solution $\bar u(\x(s),T)$, along the first diagonal (segment $[\x(0),\x(\sqrt{2})]$  with curvilinear abscissae $s$ such that $\x(0) = (0,0)$ and $\x(\sqrt{2}) = (1,1)$), with $\theta = 0.5$. Red line and plus: exact solution; blue line and circles: approximate solution with Scheme \eqref{eq:gradscheme}; green line and squares: approximate solution with upstream scheme.  Right: $|\Pi_\disc u^{(N)}(\x(s))-\bar u(\x(s),T)|$, blue line and circles: approximate solution with Scheme \eqref{eq:gradscheme}; green line and squares: approximate solution with upstream scheme. Computations done using \texttt{mesh1\_4} and $\dt = 0.0125$.}}
\end{figure}

\begin{table}[ht!]
\begin{center}
\begin{tabular}{|c|c|c|c|c|c|c|c|}
\hline 
$h$ & \texttt{errl2} & rate & \texttt{errl1} & rate &\texttt{errl$\infty$} & rate \\
\hline
 0.250 &  4.96E-02 &    -   &  4.34E-02 &    -   &    0.138 &   - \\
\hline 
 0.125 &  1.82E-02 &  1.44 &  1.42E-02 &  1.61 &    7.17E-02 &  0.94 \\
\hline 
 0.062 &  5.89E-03 &  1.62 &  4.26E-03 &  1.73 &    3.59E-02 &  0.99 \\
\hline 
 0.031 &  1.81E-03 &  1.70 &  1.16E-03 &  1.87 &    1.78E-02 &  1.01 \\
\hline 
 0.016 &  5.51E-04 &  1.71 &  3.06E-04 &  1.92 &    8.86E-03 &  1.01 \\
\hline
\end{tabular}
\end{center}
\caption{\label{tab:cvfe-cent-div} {Case 2, results with Scheme \eqref{eq:gradscheme}, using the CVFE method, $\dt = 0.4 h$ }}
\end{table}

\begin{table}[ht!]
\begin{center}
\begin{tabular}{|c|c|c|c|c|c|c|c|}
\hline 
$h$ & \texttt{errl2} & rate & \texttt{errl1} & rate &\texttt{errl$\infty$} & rate \\
\hline
 0.250 &  5.37E-02 & - &  5.01E-02   & - &   9.25E-02  & -  \\
\hline 
 0.125 &   2.88E-02 & 0.90 & 2.59E-02  & 0.95 &  5.85E-02 & 0.66\\
\hline 
 0.062 &    1.57E-02 & 0.87 &  1.35E-02 & 0.94 &   3.36E-02 & 0.80\\
\hline 
 0.031 & 8.55E-03 & 0.88 &  6.94E-03 & 0.96 &  2.21E-02 & 0.60 \\
\hline 
 0.016 &   4.57E-03 & 0.90 &  3.53E-03 & 0.98 &  1.50E-02 & 0.55 \\
\hline
\end{tabular}
\end{center}
\caption{\label{tab:cvfe-ups-div} {Case 2, results with the upstream scheme, using the CVFE method, $\dt = 0.4 h$ }}
\end{table}

\section{Conclusion}

We designed a numerical scheme, based on the Gradient Discretisation Method, for linear advection equations. 
The approximation is built on a skew-symmetric formulation of the advective terms, which enables estimates and a complete proof of convergence without additional regularity on the solution.  
The abstract notion of the size of a GD is used in both the design of the scheme and in the characterisation of the properties of the GDM. 
We note that this size of GD is defined purely using the underlying abstract spaces and operators; although linked to the mesh size for mesh-based schemes, it can also be fully defined for meshless methods.

The analysis carried out in this paper may also lead to the development and analysis of novel GDM-based schemes for coupled hyperbolic-parabolic problems.

\bibliographystyle{abbrv}
\bibliography{conv_gdm}

\end{document}

%% file: fig-ml.pdf_t
\begin{picture}(0,0)%
\includegraphics{fig-ml.pdf}%
\end{picture}%
\setlength{\unitlength}{3947sp}%
\begingroup\makeatletter\ifx\SetFigFont\undefined%
\gdef\SetFigFont#1#2#3#4#5{%
  \reset@font\fontsize{#1}{#2pt}%
  \fontfamily{#3}\fontseries{#4}\fontshape{#5}%
  \selectfont}%
\fi\endgroup%
\begin{picture}(2846,2636)(552,-3008)
\put(2758,-2021){\makebox(0,0)[lb]{\smash{{\SetFigFont{10}{12.0}{\familydefault}{\mddefault}{\updefault}{\color[rgb]{0,0,0}$K$}%
}}}}
\put(1793,-1852){\makebox(0,0)[lb]{\smash{{\SetFigFont{10}{12.0}{\familydefault}{\mddefault}{\updefault}{\color[rgb]{0,0,0}$C_\vertex$}%
}}}}
\put(2157,-2180){\makebox(0,0)[lb]{\smash{{\SetFigFont{10}{12.0}{\familydefault}{\mddefault}{\updefault}{\color[rgb]{0,0,0}$\vertex$}%
}}}}
\put(618,-495){\makebox(0,0)[lb]{\smash{{\SetFigFont{10}{12.0}{\familydefault}{\mddefault}{\updefault}{\color[rgb]{0,0,0}$\dr\O$}%
}}}}
\put(2017,-528){\makebox(0,0)[lb]{\smash{{\SetFigFont{10}{12.0}{\familydefault}{\mddefault}{\updefault}{\color[rgb]{0,0,0}$\vertex'$}%
}}}}
\put(1594,-735){\makebox(0,0)[lb]{\smash{{\SetFigFont{10}{12.0}{\familydefault}{\mddefault}{\updefault}{\color[rgb]{0,0,0}$C_{\vertex'}$}%
}}}}
\end{picture}%

%% file: fig-mlpunnc.pdf_t
\begin{picture}(0,0)%
\includegraphics{fig-mlpunnc.pdf}%
\end{picture}%
\setlength{\unitlength}{3947sp}%
\begingroup\makeatletter\ifx\SetFigFont\undefined%
\gdef\SetFigFont#1#2#3#4#5{%
  \reset@font\fontsize{#1}{#2pt}%
  \fontfamily{#3}\fontseries{#4}\fontshape{#5}%
  \selectfont}%
\fi\endgroup%
\begin{picture}(2841,1761)(4427,-2747)
\put(5825,-1574){\makebox(0,0)[lb]{\smash{{\SetFigFont{10}{12.0}{\familydefault}{\mddefault}{\updefault}{\color[rgb]{0,0,0}$C_\edge$}%
}}}}
\put(6011,-1775){\makebox(0,0)[lb]{\smash{{\SetFigFont{10}{12.0}{\familydefault}{\mddefault}{\updefault}{\color[rgb]{0,0,0}$\edge$}%
}}}}
\put(4940,-2577){\makebox(0,0)[lb]{\smash{{\SetFigFont{10}{12.0}{\familydefault}{\mddefault}{\updefault}{\color[rgb]{0,0,0}$\edge'$}%
}}}}
\put(4560,-1123){\makebox(0,0)[lb]{\smash{{\SetFigFont{10}{12.0}{\familydefault}{\mddefault}{\updefault}{\color[rgb]{0,0,0}$\dr\O$}%
}}}}
\put(5069,-2302){\makebox(0,0)[lb]{\smash{{\SetFigFont{10}{12.0}{\familydefault}{\mddefault}{\updefault}{\color[rgb]{0,0,0}$C_{\edge'}$}%
}}}}
\end{picture}%